\documentclass{amsart}
\usepackage{amssymb}
\usepackage{amsbsy, amsthm, amsmath,amstext, amsopn}
\usepackage[all]{xy}
\usepackage{amsfonts}
\usepackage{amscd}

\theoremstyle{definition}
\newtheorem{theorem}{Theorem}[section]
\newtheorem{prop}[theorem]{Proposition}
\newtheorem{lemma}[theorem]{Lemma}

\newtheorem{cor}[theorem]{Corollary}

\theoremstyle{definition}
\newtheorem{conj}[theorem]{Conjecture}
\newtheorem{definition}[theorem]{Definition}

\theoremstyle{remark}
\newtheorem{rmk}[theorem]{Remark}

\newtheorem{note}[theorem]{Note}

\newcommand{\F}{\mathbb{F}}
\newcommand{\Z}{\mathbb{Z}}
\newcommand{\Q}{\mathbb{Q}}
\newcommand{\R}{\mathbb{R}}
\newcommand{\Zt}{\mathbb{Z}^t}
\newcommand{\Qt}{\mathbb{Q}^t}
\newcommand{\Rt}{\mathbb{R}^t}
\newcommand{\A}{\mathcal{A}}
\newcommand{\X}{\mathcal{X}}

\newcommand{\C}{\mathbb{C}}
\newcommand{\vG}{G^\vee}
\newcommand{\cO}{\mathcal{O}}
\newcommand{\cK}{\mathcal{K}}
\newcommand{\Gr}{\mathrm{Gr}}
\newcommand{\Fl}{\mathrm{Fl}}

\newcommand{\val}{\operatorname{val}}
\newcommand{\Spec}{\operatorname{Spec}}
\newcommand{\Conf}{\operatorname{Conf}}

\newcounter{Qcount}
\begin{document}

\stepcounter{Qcount}

\title{Higher Laminations and Affine Buildings}

\author{Ian Le}
\address{Department of Mathematics\\ University of Chicago\\Chicago, IL 60637}
\email{ile@math.uchicago.edu}

\begin{abstract}
We give a Thurston-like definition for laminations on higher Teichmuller spaces associated to a surface $S$ and a semi-simple group $G$ for $G=SL_m$ and $PGL_m$. The case $G=SL_2$ or $PGL_2$ corresponds to the classical theory of laminations on a hyperbolic surface. Our construction involves positive configurations of points in the affine building. We show that these laminations are parameterized by the tropical points of the spaces $\X_{G,S}$ and $\A_{G,S}$ of Fock and Goncharov. Finally, we explain how the space of projective laminations gives a compactification of higher Teichmuller space.
\end{abstract}

\maketitle

\tableofcontents

\section{Introduction}

Higher Teichmuller theory studies the space of representations of $\pi_1(S)$, the fundamental group of a surface $S$, into a split-real group $G(\mathbb{R})$. This space is qualitatively quite different from the space of representations of $\pi_1(S)$ into a complex group or a compact group.  For example, the space of representations into a split-real group has many components.

Hitchin \cite{H} showed that one of these components is contractible, and behaves much like classical Teichmuller space. Classical Teichmuller space is obtained when one considers this component in the special case where $G=SL_2$. Hitchin's approach was analytic, and involved the study of Higgs bundles on Riemann surfaces.

More recently, Fock and Goncharov discovered a completely different approach to higher Teichmuller theory \cite{FG1}, which is more algebraic, combinatorial and explicit. For a surface $S$ with boundary and possibly marked points on the boundary, they look at $G(\mathbb{R})$-local systems on the surface with the extra data of a framing of the local system at the boundary of $S$. (With some modification the theory extends to the case of local systems on closed surfaces.) Fock and Goncharov define a pair of auxiliary moduli spaces $\X_{G,S}$ and $\A_{G,S}$ which are related to the space $\mathcal{L}_{G,S}$ of local systems. The two spaces differ in the type of framing of the local system the boundary.

An unusual feature of the spaces $\X_{G,S}$ and $\A_{G,S}$ (and the reason for their importance) is that each has an atlas of coordinate charts such that all transition functions involve only addition, multiplication and division. In other words, these spaces have a {\it positive atlas} and may be called {\it positive varieties}. The main ingredient in the construction of this positive atlas of coordinate charts is Lusztig's theory of total positivity. In fact, when $G=SL_m$ or $PGL_m$, they show more: these spaces have cluster-like structures on their ring of functions, and thus can be viewed as cluster varieties.

Because $\X_{G,S}$ and $\A_{G,S}$ are positive varieties, it makes sense to take the positive $\R_{>0}$ points of these moduli spaces. With some work, Fock and Goncharov show that taking the positive points gives an algebro-geometric description of the component of representations studied by Hitchin. The theory gives an explicit parameterization of positive representations, and it becomes manifest that the space of representations is contractible. Their theory has many interesting features: there is a close connection to cluster algebras; their moduli spaces can be quantized; the moduli spaces come in dual pairs which are a manifestation of Langlands duality.

The theory has geometric consequences as well. For example, Fock and Goncharov show that the corresponding representations are discrete and faithful (this was also shown by Labourie and Guichard for $G=SL_m$ using different methods). Another surprising consequence of their approach is that they can completely recover Thurston's theory of measured laminations.

In the 70's and 80's, Thurston invented the theory of laminations in a completely different context--the geometry and topology of two- and three-dimensional manifolds. Laminations arose from both the study of geodesics on hyperbolic surfaces and the study of Riemann surfaces via quadratic differentials. They gave a tool to analyze how hyperbolic/Riemann surfaces could degenerate, and gave a mapping class group-equivariant compactification of Teichmuller space.

Fock and Goncharov show that Thurston's space of measured laminations arises in a completely different way: by taking the tropical points of the $PGL_2$- and $SL_2$- moduli spaces $\X_{PGL_2,S}$ and $\A_{SL_2,S}$. This is surprising, as it gives an algebraic description of an object of geometric origin.

Moreover, by analogy, they define higher laminations as the tropical points of higher Teichmuller space. However, for groups of higher rank, a more concrete definition in the spirit of Thurston has remained elusive. Such a definition would confirm that the definition of Fock and Goncharov is the correct one, and clarify the bridge between ideas from geometric topology and the study of $G$-local systems on surfaces.

In this paper, we give a completely explicit definition of higher laminations for the space of framed $G(\mathbb{R})$-local systems on a surface $S$ and show that it coincides with the tropical points of higher Teichmuller space. We expect that one can extend much of what was done in \cite{FG3} for $SL_2$ laminations--the construction of functions corresponding to laminations, duality pairings, analogues of length functions--to the case of general $G$.

Another reason our definition is the right one is that it involves affine buildings, and is in agreement with work of Morgan, Shalen, Alessandrini, and Parreau \cite{MS}, \cite{A}, \cite{P}. Morgan and Shalen gave general procedures for compactifying $SL_2(\R)$ and $SL_2(\C)$ character varieties of two- and three-manifolds using spaces of $R$-trees. They show that this theory is equivalent to the theory of laminations (laminations and $R$-trees are, in the appropriate sense, dual to each other).

Parreau used different techniques (ultra-filters and Gromov-Haussdorf convergence) to prove a much more general result: that spaces of representations of an finitely generated group into a semi-simple Lie group have a compactification consisting of actions on buildings. Though very general, Parreau's results have some drawbacks: 1) they are not very explicit 2) the topology of these compactifications is difficult to understand and 3) it is difficult to control the types of actions of buildings that arise. Alessandrini has also shown that there exist compactifications of character varieties by actions on buildings using techniques of tropical geometry (as in Morgan and Shalen) to get a handle on these compactifications. Our approach fits in this line of work and gives a topologically simple and explicit tropical compactification of higher Teichmuller space.

Our main results are as follows:

\begin{theorem} Let $S$ be a (hyperbolic) surface with marked points, and let $C$ be its cyclic set at $\infty$. Associated to any tropical point of $\A_{G,S}(\Zt)$ there is an $\A$-lamination: a $\pi_1(S)$-equivariant virtual positive configuration of points in the affine building of $G$ parameterized by $C$. This virtual positive configuration is unique up to equivalence.

Analogously, associated to any tropical point of $\X_{G,S}(\Zt)$ there is an $\X$-lamination: a positive configurations of cones in the affine building where the cones are parameterized by every finite set of the set $C$, compatible under restriction from one finite set to another, and equipped with an action of $\pi_1(S)$ on these configurations of cones. This positive configuration of cones is unique up to equivalence.

The space of projectived laminations provides a spherical boundary for the corresponding higher Teichmuller spaces $\A_{G,S}(\R_{>0})$ and $\X_{G,S}(\R_{>0})$.
\end{theorem}

Higher laminations can be thought of as an analogue of $R$-trees that relates the combinatorics of affine buildings and representation theory on the one hand (objects like hives, honeycombs and Satake fibers), and degenerations of bundles and geometric structures on the other.

Our result can be viewed in the context of the duality between the $\X$ and $\A$ spaces for Langlands dual groups. It gives strong evidence for some of the duality conjectures of Fock and Goncharov \cite{FG2}.

We expect our theory of laminations on higher Teichmuller spaces to have lots of applications, and we list a few here.

\begin{enumerate}
\item The general philosophy of cluster algebras and the duality conjectures of Fock and Goncharov lead us to expect that the cluster complex associated to $\A_{G,S}$ embeds inside the space of laminations for the space $\X_{\vG,S}$, where $\vG$ is the Langlands dual of $G$. This gives a parameterization of all cluster variables and all clusters. This is of interest because higher Teichmuller spaces give many examples of cluster algebras of rather general (and sometimes mysterious) type and also include many well-studied cases (for example, most finite mutation type cluster algebras as well as the elliptic $E_7$ and $E_8$ algebras).

\item In fact, one expects more: laminations should parameterize atomic/canonical bases for the cluster algebras. These are of interest in physics, where they correspond with line operators \cite{GMN1}. Higher laminations are also related to the spectral networks: in \cite{GMN1}, the authors conjecture that there are spectral networks associated to each cluster in the cluster algebra and that passage through ``saddle connections'' corresponds with mutation in cluster algebras. This would mean that the space of higher laminations, whose piecewise-linear structure encodes the cluster complex, should be in bijection with spectral networks. There is strong evidence of this in the case $G=SL_2$ from the Hubbard-Mazur theorem. A more precise relationship is conjectured in \cite{KNPS}: they say that spectral networks arise from $\pi_1$-equivariant harmonic maps to buildings. We believe that the correct buildings to map to come from our construction of higher laminations.

Recent work on canonical bases in cluster algebras includes \cite{GHKK}. Another interesting application of spectral networks and cluster algebras in mirror symmetry has been the work of \cite{BS}, which relates spaces of quadratic differentials and spaces of stability conditions.

\item A better understanding of the structure of the cluster algebra on the space $\A_{G,S}$ should in particular lead to a better understanding of its symmetries. These symmetries are expected to be a higher analogue of the mapping class group. One of our motivations for a geometric definition of laminations was to obtain some understanding of the higher mapping class group. It is possible that higher laminations can be used as a tool to study of the higher mapping class group and its dynamics in the way that laminations are used to study the mapping class group.

\item One hopes to construct length functions associated to laminations on higher Teichmuller space. One approach might be to identify them with asymptotics of canonical functions on higher Teichmuller space. This could lead to a better understanding of the geometric structures that higher Teichmuller space parameterizes.

\item The study of positive configurations of points inside the affine building for $G$ is related to counting tensor product multiplicities for representations of $\vG$. Positive configurations of points inside the affine building are in bijection with components of Satake fibers (this follows from work of Goncharov-Shen \cite{GS} building on work of Kamnitzer \cite{K}, though their perspective is slightly different; we explain the relationship in section 5.7).

\item Positive configurations of points in the affine building are new objects in the geometry of the affine building. They are fairly rigid, they can be parameterized, and they are of a tropical nature, unlike general configurations of points in the affine buliding. There should be a duality pairing between positive configurations for Langlands dual groups which encapsulates rich geometric structure.
\end{enumerate}

Let us summarize the contents of this paper. In section 2, we begin by reviewing the constructions of higher Teichmuller spaces in \cite{FG1}. We will adapt some of the exposition from their introduction. We will give the correct definition of the space $\A_{G,S}$, correcting a minor error in \cite{FG1}. In section 3 we discuss some generalities on tropical points of positive varieties, spelling out some ideas that are implicit in \cite{FG1}, and related to more standard ideas in tropical geometry \cite{M}, \cite{SW}. In section 4 we review the definition and basic properties of the affine Grassmanian and affine buildings. In section 5.1 we give a conceptual outline of the definition/construction of virtual positive configurations of points in the affine building, the central object by which we define higher laminations. This summarizes sections 5.4-5.6 which form the heart of the paper. Section 5.2 lays the groundwork for the complete definition of virtual positive configurations in section 5.3. In sections 5.4-5.6, we then explain the construction of positive configurations. The use of cluster co-ordinates turns out to be important here. This leads up to Theorem~\ref{tropical points}, the central result of this paper, describing $\A$-laminations on the disc. From here we deduce how to define $\A$-laminations on any surface. In section 5.7 we discuss the relationship of our work with \cite{GS}, and derive the hive inequalities from our approach. In section 6, we give the analogous result for $\X$-laminations, and give a proposal for an extension of the definition to closed surfaces. Finally, in section 7, we describe an application of the theory: a spherical compactification of higher Teichmuller space as a closed ball such that the action of the (higher) mapping class group extends to the boundary.

\medskip
\noindent{\bf Acknowledgments} I would like to thank Vladimir Fock for his generosity in clarifying the ideas of \cite{FG1} and \cite{FG2}, and also pointing out that the study of higher laminations was an interesting problem. I thank Francois Labourie for helpful conversations and encouragement. My ideas are very indebted to Joel Kamnitzer's work, and he helped me understand the relationship of his work to my own. My thinking about cluster algebras was very influenced by Lauren Williams, Greg Musiker and David Speyer. Finally, my advisor David Nadler has continually been a valuable sounding board for my ideas, and I am grateful for his support and encouragement.

\section{Background}

\subsection{Setup}

Let $S$ be a compact oriented surface, with or without boundary, and possibly with a finite number of marked points on each boundary component. We will refer to this whole set of data--the surface and the marked points on the boundary--by $S$. We will always take $S$ to be hyperbolic, meaning it either has negative Euler characteristic, or contains enough marked points on the boundary (in other words, we can give it the structure of a hyperbolic surface such that the boundary components that do no contain marked points are cusps, and all the marked points are also cusps).

Let $G$ be a semi-simple algebraic group. When $G$ is adjoint, i.e., has trivial center (for example, when $G=PGL_m$), we can define a higher Teichmuller space $\X_{G,S}$. On the other hand, for $G$ simply-connected  (for example, when $G=SL_m$), we can define the higher Teichmuller space $\A_{G,S}$. They will be the space of local systems of $S$ with structure group $G$ with some extra structure of a framing of the local system at the boundary components of $S$. Alternatively, these spaces describe homomorphisms of $\pi_1(S)$ into $G$ modulo conjugation plus some extra data.

When $S$ does has at least one hole, the spaces $\X_{G,S}$ and $\A_{G,S}$ have a distinguished collection of coordinate systems, equivariant under the action 
of the mapping class group of $S$. Using an elaboration of Lusztig's work on total positivity, one can show that all the transition functions between these coordinate systems are subtraction-free, and give a {\it positive atlas} on the corresponding moduli space. This positive atlas gives the spaces $\X_{G,S}$ and $\A_{G,S}$ the structure of a {\it positive variety}.

If $X$ is a positive variety (for example, $X=\A_{G,S}$ or $\X_{G,S}$), we can take points of $X$ with values in any {\it semifield}, i.e. in any set equipped with operations of addition, multiplication and division, such that these operations satisfy their usual properties (the most important and non-trivial being distributivity). For us, the important examples of semifields will be the positive real numbers $\R_{>0}$; any tropical semifield; and the semifield which interpolates between these two: the field of formal Laurent series over $\R$ with positive leading coefficient, which we denote $\cK_{>0}$. The tropical semifields $\Z^t, \Q^t, \R^t$ are obtained from $\Z, \Q, \R$ by replacing the operations of multiplication, division and addition by the operations of addition, subtraction and taking the maximum, respectively.

Taking $\R_{>0}$-points of the spaces $\X_{G,S}$ and $\A_{G,S}$ allows us to recover higher Teichmuller spaces. The proof of this takes up the bulk of \cite{FG1}. This space consists of the real points of $\X_{G,S}$ and $\A_{G,S}$ whose coordinates in one, and hence in any, of the constructed coordinate charts are positive. 

The existence of these extraordinary positive co-ordinate charts depends on G. Lusztig's theory of positivity in semi-simple Lie groups \cite{Lu}, \cite{Lu2}, and is a reflection of the cluster algebra structure of the ring of functions on these spaces.

\subsection{Definition of the spaces $\X_{G,S}$ and $\A_{G,S}$}

The data of a framing of a local system involves the geometry of the flag variety associated to a group. Let $B$ be a Borel subgroup, a maximal solvable subgroup of $G$. Then ${\mathcal B} = G/B$ is the flag variety. Let $U := [B,B]$ be a  maximal unipotent subgroup in $G$. 

Let ${\mathcal L}$ be a $G$-local system on $S$. For any space $X$ on equipped with a $G$-action, we can form the associated bundle ${\mathcal L}_{X}$. For $X=G/B$ we get the associated flag bundle ${\mathcal L}_{\mathcal B}$, and for $X=G/U$, we get the associated principal flag bundle ${\mathcal L}_{\mathcal A}$. Then we will call $\A = G/U$ the ``principal affine space'' (sometimes also referred to as the ``base affine space''). We will refer to elements of $\A$ as ``principal flags.''

\begin{definition}
A {\it framed  $G$-local system on $S$} is a pair $({\mathcal L}, \beta)$, where  ${\mathcal L}$ is a  $G$-local system on $S$, and  $\beta$ a flat section of the restriction of ${\mathcal L}_{\mathcal B}$ to the punctured boundary of $S$.

The space ${\mathcal X}_{G, S}$ is the moduli space of framed $G$-local systems on $S$. 
\end{definition}

The definition of the space ${\mathcal A}_{G, S}$ is slightly more complicated. We first will need to define the notion of a twisted local system.

Let $G$ be simply-connected. Then the maximal length element $w_0$ of the Weyl group of $G$ has a natural lift to $G$, denoted $\overline w_0$. Let $s_G:= {\overline w}^2_0$. It turns out that $s_G$ is in the center of $G$ and that $s^2_G =e$. Depending on $G$, $s_G$ will have order one or order two. For example, for $G = SL_{2k}$, $s_G$ has order two, while for $G = SL_{2k+1}$, $s_G$ has order one.

The fundamental group $\pi_1(S)$ has a natural central extension by $\Z / 2\Z$. We see this as follows. For a surface $S$, let $T'S$ be the tangent bundle with the zero-section removed. $\pi_1(T'S)$ is a central extension of $\pi_1(S)$ by $\Z$:

$$\Z \rightarrow \pi_1(T'S) \rightarrow \pi_1(S).$$

The quotient of $\pi_1(S)$ by the central subgroup $2 \Z \subset \Z$, gives ${\overline \pi}_1(S)$ which is a central extension of $\pi_1(S)$ by $\Z / 2\Z$:

$$\Z / 2\Z \rightarrow {\overline \pi}_1(S) \rightarrow \pi_1(S).$$

Let $\sigma_S \in {\overline \pi}_1(S)$ denote the non-trivial element of the center.

A {\it twisted $G$-local system} is a representation ${\overline \pi}_1(S)$ in $G$ such that $\sigma_S$ maps to $s_G$. Such a representation gives a local system on $T'S$.

Now we must describe the framing data for a twisted local system. Let ${\mathcal L}$ be a twisted $G$-local system on $S$. Such a twisted local system gives an associated principal affine bundle ${\overline {\mathcal L}}_{\mathcal A}$ on the punctured tangent bundle $T'S$. For any boundary component of $S$, we will construct sections of the punctured tangent bundle above these boundary components. Given any boundary component, consider the outward pointing unit tangent vectors along the boundary--this gives a section of the punctured tangent bundle above each boundary component of $S$. We get a bunch of loops and arcs in $T'S$ lie over the boundary of $S$. Call this the {\em lifted boundary}.

\begin{definition}

A {\it decorated $G$-local system} on $S$ consists of $({\mathcal L}, \alpha)$, where ${\mathcal L}$ is a twisted local system on $S$ and $\alpha$ is a flat section of the restriction of ${\overline {\mathcal L}}_{\mathcal A}$ to the lifted boundary.

The space ${\mathcal A}_{G, S}$ is the moduli space of decorated $G$-local systems on $S$.
\end{definition}

Note that in the case where $s_G=e$, a decorated local system is just a local system on $S$ along with a flat section of ${\mathcal L}_{\mathcal A}$ restricted to the boundary. One can generally pretend that this is the case without much danger.

\subsection{Relation to configurations of flags}

The positive co-ordinate systems on ${\mathcal X}_{G, S}$ and ${\mathcal A}_{G, S}$ arise by rationally identifing them with spaces of configurations of flags. We will only outline this part of the story.

Let $S$ be a hyperbolic surface (a surface with negative Euler characteristic or with enough marked points on the boundary). Choose some hyperbolic structure on $S$ such that the boundary components that do no contain marked points are cusps, and all the marked points are also cusps. The particular choice of hyperbolic structure will turn out not to matter. Then the universal cover of $S$ will be a subset of the hyperbolic plane, and all these cusps will lie at the boundary at infinity of the hyperbolic plane. These cusps form a set $C$ that has a cyclic ordering. $C$ also carries a natural action of $\pi_1(S)$. The action of $\pi_1(S)$ preserves the cyclic ordering on $C$. The set $C$ with its cyclic order is independent of our choice of hyperbolic structure on $C$.

A $\pi_1(S)$-equivariant configuration of flags (respectively principal affine flags) parameterized by $C$ is a map $\beta: C \rightarrow {\mathcal B}$ (respectively a map $\beta: C \rightarrow {\mathcal A}$) such that there is a map $\rho: \pi_1(S) \rightarrow G$ such that for $\gamma \in \pi_1(S)$,

$$\beta(\gamma \cdot c) = \rho(\gamma) \cdot c$$
for all points $c \in C$.

Starting with any point of ${\mathcal X}_{G, S}$ (respectively ${\mathcal A}_{G, S}$), we may look at the universal cover of $S$. On the universal cover, the local system becomes trivial, and the framing of the local system then gives a flag (repectively a principal affine flag) at each point of the cyclic set $C$. Thus any point in ${\mathcal X}_{G, S}$ (respectively ${\mathcal A}_{G, S}$) gives a $\pi_1(S)$ equivariant configuration of flags (respectively principal affine flags) parameterized by $C$.

\begin{theorem}
\cite{FG1} The space ${\mathcal X}_{G, S}$ has a positive atlas that comes from identifying a framed local system with a $\pi_1(S)$-equivariant positive configuration of flags parameterized by $C$.

The space ${\mathcal A}_{G, S}$ has a positive atlas that comes from identifying a decorated local system with a $\pi_1(S)$-equivariant twisted positive cyclic configuration of principal affine flags parameterized by $C$.

\end{theorem}

\begin{note} There are two things to be careful about in the above theorem. The first is that the atlas is actually on the space of positive configurations of flags parameterized by $C$ that are equivariant for {\em some} map $\rho: \pi_1(S) \rightarrow G$. It is possible, though quite unusual, that the same configuration of flags could be equivariant for more than one representation $\rho: \pi_1(S) \rightarrow G$. This is the reason that the identification of ${\mathcal X}_{G, S}$ (respectively ${\mathcal A}_{G, S}$) with positive configurations of flags (respectively principal affine flags) parameterized by $C$ is only birational. The second, less serious, caution, is that it is unclear whether the atlases cover the whole space. However, we do not require this to be the case.

We shall see that once we have taken the $\R_{>0}$-points of these spaces, both these problems magically disappear--the positive points are sufficiently generic that each point of ${\mathcal X}_{G, S}$ (respectively ${\mathcal A}_{G, S}$) corresponds to a unique configuration of flags, and each coordinate chart completely covers the positive part of the space (though these facts require some work to prove).
\end{note}

In particular, when  $S$ is a disk with marked points on the boundary we get moduli spaces of configurations of points 
in the flag variety ${\mathcal B}: = G/B$ and twisted configurations of points of the principal affine variety ${\mathcal A}:= G/U$. For more details, see \cite{FG1}
 
\section{The tropical points of higher Teichmuller space}

In the previous section, we explained the construction of higher Teichmuller space. We will now give one construction of its tropical points. The main task of this paper will be to use this construction to give a concrete description of higher laminations.

We give a simple way to construct tropical points of any positive variety. Let $\cK$ be the field of formal Laurent series over $\R$, $\R((t))$. This ring has a natural valuation $val: \cK \rightarrow \Z.$ The goal will be to understand tropical $\Z^t$-points of a variety via valuations of its $\cK$-points.

Let us consider the positive semifield $\cK_{>0}$ which consists of those Laurent series with positive leading coefficient. Let $X$ be any positive variety, in other words a variety with an atlas of charts such that all transition functions involve only multiplication, division and addition (for example $\X_{G,S}$ or $\A_{G,S}$). We may then consider the $\cK_{>0}$ points of these varieties.

Let $x \in X(\cK_{>0})$. Then there is a corresponding tropical point $x^t$ of the space $X(\Zt)$. This point $x^t$ is characterized by the property that if $f$ is one of the positive coordinates of a positive chart (or more generally any positive function), then $$f(x^t)=-\val f(x).$$
In other words, we specify the tropical coordinates of $x^t$ in each chart as being negative of the valuation of the coordinate of $x$. To see that $x^t$ is well-defined, we only need to check that under a change of coordinate charts, the functions $-\val f(x)$ transform tropically. However, this is clear, because all transition functions between coordinate charts involve only multiplication, division and addition, and we have $$-\val(f(x)g(x))=-\val(f(x))-\val(g(x)),$$ $$-\val(f(x)/g(x))=-\val(f(x))+\val(g(x)),$$ $$-\val(f(x)+g(x))=max\{-\val(f(x)),-\val(g(x))\}$$ whenever $f$ and $g$ are functions coming from the coordinate charts. The last equality holds because both $f(x)$ and $g(x)$ have positive leading coefficient.

In other words, the negative valuations of $\cK_{>0}$-points of a positive variety automatically satisfy the tropical relations. Thus we get a map $-\val:  X(\cK_{>0}) \rightarrow X(\Zt)$. The map is surjective: in any coordinate chart, we may specify the valuations of coordinates of a point of  $X(\cK_{>0})$ as we wish. Because all transition functions between charts are invertible, specifying the coordinates in one chart is the same as specifying the coordinates in every chart.

Thus, in order to understand tropical points of $\X_{G,S}$ or $\A_{G,S}$, we must analyze the fibers of this map and see what the points in one fiber have in common. In other words, we would like to isolate what invariant information is contained in the tropical functions. Our goal will be to show that he piece-wise linear combinatorics of the affine building is exactly what is captured by the tropical coordinates.

Our goal will be acheived in the case where $G=SL_m$. Many of the steps will have clear generalizations to general groups (for example, the definition of laminations in terms of affine buildings); we will note those steps which do not extend as straightforwardly. We hope to treat in a future paper the case of a general semi-simple Lie group, for which we believe the best approach would be to explicitly construct cluster coordinates on Teichmuller spaces associated to these groups. These coordinates would be analogous to the ``canonical coordinates'' of Fock and Goncharov on Teichmuller spaces for $SL_m$. Henceforth we will be concerned primarily with the case $G=SL_m$ or $G=PGL_m$.

If instead of considering $K$, we consider the ring of Laurent series over $t^\lambda$ where $\lambda$ is allowed to vary in $\Z, \Q$ or $\R$, we get obtain different types of laminations with coefficients in  $\Zt, \Qt$ or $\Rt$.

\subsection{Limit sets and compactifications}

Here is one useful way to think about tropical points. $\Spec \cK$ can be thought of as the punctured formal neighborhood of the origin in the real line. (One says something similar about $\F((t))$ for any field $\F$.) Then $\Spec \cK_{>0}$ is the positive part of the punctured formal neighborhood of the origin. This punctured formal neighborhood looks like an infinitesimal path. We may then interpret $X(\cK_{>0})$ as the space of infinitesimal paths in the positive part of $X$, $X(\R_{>0})$. The path goes towards infinity in $X(\R_{>0})$ exactly when the point $x \in X(\cK_{>0})$ correponding to this path has non-zero valuation; when $x$ has valuation $0$, we have that $x$ actually comes from a point in $X(\R[[t]])$ and the path has an endpoint in $X(\R_{>0})$. Paths which have the same valuation are asymptotic to each other. Paths that have proportional valuations go towards the same point at infinity, but at different rates. Thus we see why projectivized tropical points should give a boundary to the positive space $X(\R_{>0})$.

More precisely, paths which have proportional valuations approach the same point at infinity of the logarithmic limit set, first defined by Bergman in \cite{B}. Projectivized tropical points form the boundary of the logarithmic limit set. The logarithmic limit set gives us a tropical compactification of $X$. We explain this further in Section 7.

One may replace the ring $\Spec \cK_{>0}$ by some variants, such as convergent power series, germs of convergent power series, Puiseux series, etc., in which case the infinitesimal paths in the above discussion will be replaced by actual paths in  $X(\R_{>0})$. Again, these paths will approach some boundary point in the tropical compactification. This point of view is explained thoroughly in \cite{A2}.

\subsection{Other approaches to tropicalization}

We should say a word about why we use the semi-field $\cK_{>0}$ in the above discussion, as opposed to a field like $\R((t))$ or $\C((t))$. The reason is that we seek a tropicalization and compactification of the space $X(\R_{>0})$. Taking valuations of $\R((t))$- or $\C((t))$-points of $X$ gives tropical points of $X(\R)$ or $X(\C)$. If we were looking to tropicalize or compactify the moduli space of representations of $\pi_1(S)$ into $G(\R)$ and not just the positive component, we would consider valuations of $\R((t))$-points. Similarly, if we were looking to tropicalize or compactify the moduli space of representations of $\pi_1(S)$ into $G(\C)$, we would consider valuations of $\C((t))$-points.

Let use elaborate on this. Our approach of taking valuations of $\cK_{>0}$-points is different from the usual definition of tropicalization. Tropical geometry began with the work of Bergman, \cite{B}, and later the work of Bieri and Groves, \cite{BG}. There are several ways to define the tropicalization of a variety $X$, and a comparison of various definitions can be found in the book of Sturmfels and Maclagan (\cite{M}).

Here is one definition of the tropicalization of a variety $X$. The tropicalization of $X$ consists of the closure of the points obtained by taking valuations of the $\C((t^{1/n}))$-points of $X$ over all $n \in \mathbb{N}$. Equivalently, it is the closure of the points obtained by taking valuations of complex Puiseux-series-valued points of $X$. The integral points of the tropicalization come from taking valuations of $\C((t))$-points of $X$. The projectivization of the tropicalization of $X$ forms the boundary of the logarithmic limit set of $X(\C)$. A proof of this is outlined in \cite{M}, with one important step of the proof found in \cite{Pa}. Another good survey which explains this is \cite{Gu}.

In contrast, our set $X(\Zt)$ consists of valuations of $\cK_{>0}$-points of $X$, while what we call the tropicalization, $X(\Rt)$, is the closure of $X(\Qt)$, the set of valuations of points taken in the field of Puiseux series real coefficients and positive leading coefficient. The projectivized tropicalization gives the boundary of the logarithmic limit set of $X(\R_{>0})$. The proof of this fact was given by \cite{A2}.

Thus our approach to tropicalization, following the definitions of Fock and Goncharov, where one takes values of a variety in a tropical semi-field, or equivalently takes valuations of $\R_{>0}$-points, might be more properly called the {\it positive part of the tropicalization}. However, because our primary object of study is $X(\R_{>0})$, when we talk of tropical points, we will always mean this positive part of the tropicalization. In the context of higher Teichmuller theory, this seems to be the more fundamental object; it is the object that ``knows'' about the cluster complex, and it conjecturally parameterizes canonical bases of functions and spectral networks.

For a comparison of tropical varieties and their positive parts, see work of Speyer and Williams on the tropical Grassmanian \cite{SW}. Our work is a kind of generalization of theirs, but treating only the positive parts of tropical varieties. From that paper, one sees that a tropical variety is in general much more complicated, both topologically and combinatorially, than its positive part.

\section{Background on affine buildings}

In this section we define the affine Grassmanian and affine buildings. We will follow some of the treatment of \cite{FKK}, which in many ways was inspiration for our work. We will first define the real affine Grassmanian. (The affine Grassmanian is an ind-scheme, or an inductive limit of schemes, and it can be defined over any ring, but we will make use only of its real points.) Let $G$ be a simple, simply-connected complex algebraic group and let $\vG$ be its Langlands dual group. Let $\cO = \R[[t]]$ be the ring of formal power series over $\R$ and let $\cK = \R((t))$ be its fraction field. Then 
$$\Gr = \Gr(G) = G(\cK)/G(\cO)$$ is the \emph{real affine Grassmannian} for $G$.  It can be viewed as a direct limit of real varieties of increasing dimension.

For $G=SL_m$, a point in the affine Grassmanian corresponds to a finitely generated, rank $m$ $\cO$-submodule of $\cK^m$ such that if $v_1, \dots, v_m$ are generators for this submodule, then $$v_1 \wedge \dots \wedge v_m=e_1 \wedge \dots \wedge e_m$$ where $e_1, \dots, e_m$ is the standard basis of $\cK^m$. Such full rank $\cO$-submodules are often called {\it lattices}. $G(\cK)$ acts on the space of lattices with the stabilizer of each lattice being isomorphic to $G(\cO)$, which acts by changing the basis of the submodule while leaving the submodule itself fixed. We will later make use of this interpretation.

The affine Grassmannian $\Gr$ also has a metric valued in dominant coweights: the set of two elements of $\Gr$ up to the action of $G(\cK)$ is exactly the set of double cosets $$G(\cO) \backslash G(\cK) / G(\cO).$$ These double cosets, in turn, are in bijection with the cone $\Lambda_+$ of dominant coweights of $G$. Recall that the coweight lattice $\Lambda$ is defined as $\mathrm{Hom}(\mathbf{G}_m,T)$. The coweight lattice contains dominant coweights, those coweights lying in the dominant cone. For example, for $G=GL_m$, the set of dominant coweights is exactly the set of $\mu=(\mu_1, \dots, \mu_m)$, where $\mu_1 \geq \mu_2 \geq \cdots \geq \mu_n$ and $\mu_i \in \Z$. For $G=PGL_m$, it is this same set modulo the the one-dimensional space spanned by the vector $(1,1,\dots, 1)$.

Let us explain why the set of double cosets is in bijection with the set of dominant coweights. Given any dominant coweight $\mu$ of $G$, there is an associated point $t^\mu$ in the (real) affine Grassmannian: to a coweight $\mu=(\mu_1, \dots, \mu_m)$ we associate the element of $G(\cK)$ with diagonal entries $t^{\mu_i}$, and then project to the affine Grassmanian $\Gr$ by the quotient map. Any two points $p$ and $q$ of the affine Grassmannian can be translated by an element of $G(\cK)$ to $t^0$ and $t^\mu$, respectively, for some unique dominant coweight $\mu$. This gives the identification of the double coset space with $\Lambda_+$.

Under this circumstance, we will write $$d(p,q) = \mu$$ and say that the distance from $p$ to $q$ is $\mu$. Let us collect some facts about this distance function $d$.  Note that this distance function is not symmetric; one can easily check that $$d(p,q)=-w_0 d(q,p)$$ where $w_0$ is the longest element of the Weyl group of $G$ (recall that the Weyl group acts on both the weight space $\Lambda^*$ and its dual, the coweight space $\Lambda$). However, there is a partial order on $\Lambda$ defined by $\lambda > \mu$ if $\lambda - \mu$ is positive (i.e., in the positive span of the positive co-roots). Under this partial ordering, the distance function satisfies a version of the triangle inequality. By construction, the action of $G(\cK)$ on the affine Grassmannian preserves this distance function.

We are interested in the affine Grassmanian, but not in its finer structure as a variety. In fact, we will only consider properties of the affine Grassmanian that depend on positivity and the above distance function. For this reason, we will introduce affine buildings, a combinatorial skeleton of the affine Grassmanian. Once we impose positivity conditions, all our constructions will only involve subsets of the affine building and the induced metric on them.

For any group $G$, the points of the associated affine Grassmannian $\Gr$ are a subset of the vertices of an simplicial complex called the \emph{affine building} $\Delta = \Delta(G)$. $\Delta(G)$ is associated with the extended Dynkin diagram of $G$. The simplices of this affine building correspond to parahoric subgroups of the affine Kac-Moody group $\widehat{G}$.

For this paper, we will be concerned with the affine building when $G=SL_m$ or $PGL_m$. Let us discuss these examples in more detail. The affine Grassmannian for $G=PGL_m$ consists of lattices (finitely generated, rank $m$ $\cO$-submodules of $\cK^m$) up to scale: two lattices $L$ and $L'$ are equivalent if $L=cL'$ for some $c \in \R((t))$. The set of vertices of the affine building for $PGL_m$ is precisely given by the points of the affine Grassmannian $\Gr(PGL_m)$.

For any lattices $L_0, L_1, \dots, L_k$, there is a $k$-simplex with vertices at $L_0, L_1, \dots, L_k$ if and only if (replacing each lattice by an equivalent one if necessary) $$L_0 \subset L_0 \subset \cdots \subset L_k \subset t^{-1}L_0.$$
This gives the affine building the structure of a simplicial complex. The affine building for $G=SL_n$ is the same simplicial complex, but where we restrict our attention to those vertices that come from the affine Grassmannian for $G=SL_n$.

The non-symmetric, coweight-valued metric we defined above descends from the affine Grassmannian to the affine building. The notion of a {\emph geodesic} with respect to his metric is sometimes useful. For our purposes, a geodesic in the building is a path that travels along edges in the building from vertex to vertex, such that the sum of the distances from vertex to vertex is minimal (with respect to the partial order defined above). It is a property of affine buildings that geodesics exist. Note that in general there will be many geodesics between two any points.

The building is a non-disjoint union of \emph{apartments}. One apartment is given by projecting the points of $T(\cK) \subset G(\cK)$ to $\Gr$. All other apartments are translates of this one. Each apartment, as a simplicial complex, is isomorphic to the Weyl alcove simplicial complex of $G$.

It is a fact that any two simplices of $\Delta$ of any dimension are both contained in at least one apartment. In particular, any two points in the building are contained in some apartment. Moreover, every geodesic between two points is contained inside every apartment containing those two points. Thus the set of geodesics between two points is completely determined by the distance between them.

\section{Laminations for the $\A$ space}

\subsection{Outline of the definition and construction of $\A$-laminations}

We will first consider the space of integral laminations $\A_{G,S}(\Zt)$ where $S$ is a disk with marked points on the boundary. For the purpose of orientation for the next few sections, in which we give the construction of positive virtual configurations of points in the building, we first give a conceptual outline. Nothing in this section is strictly necessary; the reader may choose to skip to section 5.2.

All the constructions here can be extended without difficulty to rational or real lamination spaces ($\A_{G,S}(\Qt)$ or $\A_{G,S}(\Rt)$). In the following, we will assume that $G=SL_m$.

We will begin by describing laminations on a disc with $n$ marked points. The arguments in \cite{FG1} show that understanding laminations on a surface reduces--via cutting, gluing and $\pi_1$-equivariance--to the case of laminations on a disc with 2, 3 or 4 marked points. We will build up to the following definition:

\begin{definition} A $G$-lamination on a disc with $n$ marked points is a virtual positive configuration of $n$ points in the affine building for $G$ up to equivalence.
\end{definition}

A configuration of $n$ points in the building is a set of $n$ labelled vertices of the affine building for $G$. We will study configurations up to equivalence. 

\subsubsection{Equivalence}\label{equivalence def} Equivalence will be the smallest equivalence relation generated by isometry and cutting and gluing. Thus we may define equivalence inductively. Let $p_1, \dots, p_n$ and $p'_1, \dots, p'_n$ be two configurations of points of the affine building. If they are to be equivalent, we first require that the pairwise distances between corresponding points are equal:

$$d(p_i,p_j)=d(p'_i,p'_j)$$
Define a {\em perimeter} of a configuration $p_1, \dots, p_n$ to be a union of some choice of geodesics between each $p_i$ and $p_{i+1}$, where indices are taken cyclically. Then, because
$$d(p_i,p_{i+1})=d(p'_i,p'_{i+1}),$$
we may choose a corresponding perimeter for $p'_1, \dots, p'_n$. Now let $a$ and $b$ be two points in the perimeter of the first configuration. Suppose that $a$ is on the geodesic between $p_i$ and $p_{i+1}$ and $b$ on the geodesic between $p_j$ and $p_{j+1}$. We may choose some geodesic between $a$ and $b$. Take the corresponding points $a'$ and $b'$ on the perimeter of $p'_1, \dots, p'_n$. Then we make a ``cut'' to form the configurations $$a, p_{i+1}, \dots, p_j, b$$ and $$b, p_{j+1}, \dots, p_i, a$$ and the corresponding configurations $$a', p'_{i+1}, \dots, p'_j, b'$$ and $$b', p'_{j+1}, \dots, p'_i, a'.$$ Then $p_1, \dots, p_n$ and $p'_1, \dots, p'_n$ are equivalent if and only if $a, p_{i+1}, \dots, p_j, b$ is equivalent to $a', p'_{i+1}, \dots, p'_j, b'$ and $b, p_{j+1}, \dots, p_i, a$ is equivalent to $b', p'_{j+1}, \dots, p'_i, a'$. Using cuts, we can reduce to the case of triangles with miniscule side lengths. Finally, we say that two such triangles are equivalent if their side lengths coincide. We will show that our definition of equivalence, when restriced to positive configurations, does not depend on the sequence of cuts in Proposition~\ref{tropical relations}. (This is no longer true if we remove the positivity assumption.) From now on, when we discuss configurations in the affine building, we will be tacitly be considering them up to equivalence.

\begin{rmk} It is likely that positivity allows us to strengthen our notion of equivalence. For example, we believe that the equivalence of two positive configurations $x_1, \dots, x_n$ and $x'_1, \dots, x'_n$ implies that their convex hulls are isometric, where the convex hull of $x_1, \dots, x_n$ is the smallest geodesically closed subset containing $x_1, \dots, x_n$. It is even possible that there is an isometry of the entire affine building which carries $x_1, \dots, x_n$ and $x'_1, \dots, x'_n$. It may be useful in applications to have a somewhat stronger notion of equivalence, though for this paper, the one given above is sufficient.
\end{rmk}

\subsubsection{Positivity} 

In this section we will try to give the idea behind the adjective ``positive.'' We will later describe how a configuration of $n$ points in $G/U((t))$ along with a set of $n$ ``large enough'' coweights gives rise to a configuration of points in the affine Grassmanian. Actually, a configuration of $n$ points in $G/U((t))$ along with a set of $n$ ``large enough'' coweights actually gives rise to a configurations of points in the affine flag variety (which we will denote $\Fl$), which is a bundle over the affine Grassmanian where the fibers of the bundle are flag varieties. We do not define the affine flag variety here, because it turns out not to be relevant for our purposes. However, we mention it because we believe that the affine flag variety will be useful for analyzing higher laminations for groups other than $SL_m$ and $PGL_m$, and even in the cases we consider, it may help some readers to keep it in mind.

We have the following maps:

$$\Conf_n G/U((t)) \times \Lambda^n \cdots \rightarrow \Conf_n \Fl \rightarrow \Conf_n \Gr \rightarrow \Conf_n \Delta(G)$$

Here the first map is only defined for a set of $n$ large enough $\Lambda^n$ (this is the reason we use the dotted arrow notation). It turns out that for any configuration in $\Conf_n G/U((t))$, there exist coweights $\mu_1, \mu_2, \dots, \mu_n$ such that whenever we have coweights $\lambda_1, \lambda_2, \dots, \lambda_n$, with $\lambda_i-\mu_i$ dominant, then $\lambda_1, \lambda_2, \dots, \lambda_n$ are large enough. Thus large enough coweights form a cone within the space of coweights.

The next step will be to describe positive configurations in $G/U((t))$. Positivity will subtly depend on the cyclic order of the $n$ points. The notion of positivity carries over to a notion of positive configurations in the affine flag variety, the affine Grassmanian and the affine building:

$$\Conf_n^+ G/U((t)) \times \Lambda^n \cdots \rightarrow \Conf_n^+ \Fl \rightarrow \Conf_n^+ \Gr \rightarrow \Conf_n^+ \Delta(G)$$

Positivity will allow us to deduce certain properties of these configurations of points in the affine building, for example the fact that tropical coordinates completely determine the configuration up to equivalence. The idea is that positivity ensures certain genericity properties of our configuration. More degenerate configurations of points in the affine building can be quite complicated, and classifying them seems to be a rather unwieldy problem \cite{CHSW}. On the other hand, positive configurations of points can be completely and explicitly understood. It will not do any harm to assume that all configurations of points henceforth are also positive.

\begin{note}\label{positive structure} Fock and Goncharov give a definition of positive configurations of points in the flag variety. Their definition should extend to a definition of positive configurations of points in the affine flag variety. (Which in turn should be related to positivity in loop groups \cite{LP}.) It would be interesting to compare our notion of positive configurations with these other notions of positivity.
\end{note}

\subsubsection{Virtual}

Our first steps towards understanding laminations involve understanding the space $\A_{G,S}(\cK_{>0})$ for $S$ a disk with $n$ marked points. This space is exactly $\Conf_n^+ G/U((t))$. Therefore we need to analyze how the above maps depend on the set of $n$ coweights in $\Lambda^n$. The subset of $\Lambda^n$ for which the maps

$$\Conf_n^+ G/U((t)) \times \Lambda^n \cdots \rightarrow \Conf_n^+ \Fl \rightarrow \Conf_n^+ \Gr \rightarrow \Conf_n^+ \Delta(G)$$
are defined is a product of cones inside each copy of $\Lambda$. Suppose that we have a configuration $F_1, F_2, \dots, F_n$ of $n$ points in $G/U((t))$ and a set of large enough coweights $\lambda_1, \lambda_2, \dots, \lambda_n$. Let this be mapped to the configuration of points $\tilde{x_1}, \tilde{x_2}, \dots, \tilde{x_n}$ in the affine flag variety, the configuration of points $x_1, x_2, \dots, x_n$ in the affine Grassmanian, and then the configuration of points $p_1, p_2, \dots, p_n$ inside the affine building.

As $\lambda_i$ varies within a cone of large enough values for which the map is defined, $\tilde{x_i}$ (respectively $x_i$ or $p_i$) varies as well. Thus each copy of $\Lambda$ moves around the corresponding point in the configuration of points in the affine flag variety (respectively the affine Grassmanian or the affine building) independently.

An alternative way to view this is that any configuration of points in $G/U((t))$ gives rise to a configuration of $n$ cones of points inside $\Fl, \Gr$ or $\Delta(G)$, where the points in these cones are parameterized by large enough coweights $\lambda_i$. 

In order to motivate the definition of \emph{virtual}, let us pause for a moment to explain the roles of the intermediate spaces $\Conf_n^+ \Fl$ and $\Conf_n^+ \Gr$.

It turns out that knowing the positive configurations of cones of points in $\Fl$ or $\Gr$ parameterized by large $\lambda_i$ is almost exactly the same information as a positive configuration of $n$ points in $G/U((t))$ (i.e., a point of $\Conf_n^+ G/U((t))$). This is, in turn, exactly the same information as a point in $\A_{G,S}(\cK_{>0})$. However, this contains too much information, as we are ultimately interested in the tropical points $\A_{G,S}(\Zt)$. Therefore, we would like to consider points of $\A_{G,S}(\cK_{>0})$ with the same valuation to be equivalent. The information contained in these valuations is exactly that which is captured by the geometry of the affine building (a precise statement can be found in ~\ref{tropical distance functions}).

For that reason, we would like to examine how a given positive configuration of flags $F_1, F_2, \dots, F_n$ (in other words, a point of $\Conf_n^+ G/U((t))$) along with a set of coweights $\lambda_1, \lambda_2, \dots, \lambda_n$ gives a configuration of points $p_1, p_2, \dots, p_n$ in the affine building $\Delta(G)$, and how the points $p_i$ move around as the $\lambda_i$ vary. 

For any dominant coweights $\mu_1, \mu_2, \dots, \mu_n$, let the positive configuration of flags $F_1, F_2,$ $\dots, F_n$ and the set of coweights $\lambda_1+\mu_1$, $\lambda_2+\mu_2, \dots,$ $\lambda_n+\mu_n$ get mapped to the configuration of points $q_1, q_2, \dots, q_n$ in $\Delta(G)$. We then would like to understand the relationship of $q_1, q_2, \dots, q_n$ to $p_1, p_2, \dots, p_n$. It turns out that up to the notion of equivalence defined above in ~\ref{equivalence def}, $q_1, q_2, \dots, q_n$ is completely determined by $p_1, p_2, \dots, p_n$ and $\mu_1, \mu_2, \dots, \mu_n$.

Thus we may view the different configurations $q_1, q_2, \dots, q_n$ that appear as the $\mu_i$ vary as being obtained by an action of the monoid $\Lambda_+^n$ on configurations of points in $\Delta(G)$. In other words, positive configurations in the affine building come equipped automatically with an action of $\Lambda_+^n$. (We will later see that there is an action of $\Lambda_+$ for every marked point and every hole on the surface $S$. Furthermore, this will extend to an action of the entire coweight lattice $\Lambda$ on the space of virtual configurations. Moreover, on a surface with holes, this extends to an action of an affine Weyl group at each puncture on the space of laminations, as we will show in future work \cite{Le}.) In order to define {\em virtual} configurations, we must understand this action.

Let us try to describe the action of $\Lambda_+^n$ intuitively. There is one factor $\Lambda_+$ acting on each of the $n$ points in the configuration. Let $p_1, p_2, \dots, p_n$ be a configuration of $n$ points. $\lambda \in \Lambda_+$ acts on the point $p_1$ by moving $p_1$ to another point $p_1'$ of the building which is distance $\lambda$ away, i.e. such that $$d(p_1, p_1')=\lambda.$$
There are of course many choices for $p_1'$; however, for the most generic choices we are moving $p_1$ a distance $\lambda$ away from all the other points in the configuration. It turns out all these choices are equivalent in the sense that the configurations $p_1', p_2, \dots, p_n$ will be equivalent for all such choices of $p_1'$.

A {\em virtual} positive configuration of points in the affine building is a set of $n$ pairs $(p_i, \lambda_i)$, $i=1, 2, \dots, n$, where $\lambda_i$ are all coweights and the $p_i$ form a positive configuration. Let $(p_i, \lambda_i)$ and $(q_i, \mu_i)$ be two virtual configurations.

Suppose that all the $\lambda_i$ and $\mu_i$ are dominant coweights. Then we may allow $(\lambda_1, \dots, \lambda_n)$ to act on $(p_1, \dots, p_n)$ and $(\mu_1, \dots, \mu_n)$ to act on $(q_1, \dots, q_n)$. Suppose the resulting configurations are $(p_1', \dots, p_n')$ and $(q_1', \dots, q_n')$. We will say that $(p_1', \dots, p_n')$  realizes the virtual configuration $(p_i, \lambda_i)$.

Then we will say that $(p_i, \lambda_i)$ and $(q_i, \mu_i)$ are equivalent virtual positive configurations if and only if $(p_1', p_2', \dots, p_n')$ and $(q_1', q_2', \dots, q_n')$ are equivalent as configurations. More generally, two configurations $(p_i, \lambda_i)$ and $(q_i, \mu_i)$ (where $\lambda_i$ and $\mu_i$ are coweights, but not necessarily dominant) are equivalent if and only if there exists $\nu_i$ such that $(p_i, \lambda_i+\nu_i)$ and $(q_i, \mu_i+\nu_i)$ are equivalent. For large enough $\nu_i$, the sets of coweights $\lambda_i+\nu_i$ and $\mu_i+\nu_i$ will be both dominant and large enough, so we can always verify whether two virtual positive configurations are equivalent.

Note that if the points $p_i$ form a positive configuration of points, there is the corresponding virtual positive configuration $(p_i,0)$. We will call such configurations ``actual'' configurations. To show that our notion of virtual configuration of points is well-defined, we need the following lemma which we shall prove after we have given a proper definition of virtual positive configurations:

\begin{lemma}\label{virtual equivalence} For any dominant coweights $\lambda_i$, the positive configurations $(p_i,0)$ and $(q_i,0)$ are equivalent if and only if $(p_i,\lambda_i)$ and $(q_i, \lambda_i)$ are equivalent.
\end{lemma}

\begin{cor} Two configurations $(p_i, \lambda_i)$ and $(q_i, \mu_i)$ (where $\lambda_i$ and $\mu_i$ are coweights, but not necessarily dominant) are equivalent if and only if for all sets of coweights $\nu_i$, $(p_i, \lambda_i+\nu_i)$ and $(q_i, \mu_i+\nu_i)$ are equivalent.
\end{cor}

We may finally refine our previous constructions. Before, we had a map 

$$\Conf_n^+ G/U((t)) \times \Lambda^n \cdots \rightarrow \Conf_n^+ \Delta(G)$$
defined for large enough coweights. Suppose that under this map we send 

$$(F_1, F_2, \dots, F_n) \times (\lambda_1, \lambda_2, \dots, \lambda_n) \rightarrow (p_1, p_2, \dots, p_n).$$

If it were possible, we would define a map from configurations of flags to configurations in the affine building
$$\Conf_n^+ G/U((t)) \cdots \rightarrow \Conf_n^+ \Delta(G),$$
that sends $(F_1, F_2, \dots, F_n)$ to the configuration in $\Delta(G)$ that comes from choosing the coweights $(0, 0, \dots, 0)$. However, this is not always possible, as $(0, 0, \dots, 0)$ may not be large enough. However, we may define a map from configurations of flags to virtual configurations of points in $\Delta(G)$
$$\Conf_n^+ G/U((t)) \rightarrow \Conf_n^{+,\textrm{vir}} \Delta(G),$$
by mapping 
$$(F_1, F_2, \dots, F_n) \rightarrow ((p_1,-\lambda_1), (p_2,-\lambda_2), \dots, (p_n,-\lambda_n)).$$

Finally, let us note that $n$ marked points on a disc come with the natural cyclic ordering. The properties of positive configurations in $G/U((t))$, $\Fl$, $\Gr$ and $\Delta(G)$ depend subtly upon this cyclic ordering of the $n$ marked points. Virtual positive configurations will attach a virtual point of the affine building to each of these $n$ marked points. These virtual positive configurations of points are precisely our notion of a higher lamination $S$ where $S$ is a disc with marked points. Our notion of higher lamination only depends on the surface $S$ and the group $G$. Rotations of a disc with marked points that permute the marked points on the boundary cyclically are a topological automorphism. Thus we expect that virtual positive configurations will have a natural cyclicity property: $(p_1,\lambda_1), \dots, (p_n,\lambda_n)$ is a positive virtual configuration if and only if every cyclic shift is a positive virtual configuration. We will see this to be the case.

This concludes the outline of the definition/construction of virtual positive configurations. We now make these definitions more precise. 

\subsection{Canonical Coordinates on Configurations in the Flag Variety and Affine Grassmanian}

Our goal in this section is to define some important functions on higher Teichmuller space and their proposed tropicalizations. These functions will be essential to all our later constructions. In all that follows, fix $G=SL_m$. We begin by recalling the canonical coordinates on a triple of principal flags for the group $SL_m$. A principal flag for $SL_m$ consists of a point in $G/U$, where $U$ is the subgroup of lower-triangular unipotent matrices; in concrete terms, we can write this as a set of $m$ vectors $v_1, \dots, v_m$ where we only care about the forms $$v_1 \wedge \dots \wedge v_k$$ for $k=1, 2, \dots, m-1$. We will require that $$v_1 \wedge \dots \wedge v_m$$ is the standard volume form. We are interested in the space of three generic flags up to the left translation action of $G$. Suppose we have three generic flags $F_1, F_2, F_3$ which are represented by $u_1, \dots, u_m$, $v_1, \dots, v_m$ and $w_1, \dots, w_m$ respectively. There is an invariant $f_{ijk}$ of this triple of flags for every triple of integers $i, j, k$ such that $i+j+k=m$ and each of $i, j, k$ is an integer strictly less than $m$. It is defined by $$f_{ijk}(F_1, F_2, F_3)=\det(u_1, u_2, \dots, u_i, v_1, v_2, \dots v_j, w_1, w_2, \dots, w_k),$$ and it is $G$-invariant by definition. Note that when one of $i, j, k$ is $0$, these functions only depend on two of the flags. We can call such functions {\em edge} functions, and the remaining functions {\em face} functions.

Given a cyclic configuration of $n$ flags, imagine the flags sitting at the vertices of an $n$-gon, and triangulate the $n$-gon. Then taking the edge and face functions on the edges and faces of this triangulation, we get a set of functions on a cyclic configuration of flags.

\begin{theorem} For any triangulation, the edge and face functions form a coordinate chart. Different triangulations yield different functions that are related to the original functions by a positive rational transformation (a transformation involving only addition, multiplication and division) \cite{FG1}.
\end{theorem}

We will now analogously define the triple distance functions $f_{ijk}^t$ on configuration of three points in the affine Grassmanian for $SL_m$. The functions $f_{ijk}^t$ are the same as the functions $H_{ijk}$, which were defined in a slightly different way in \cite{K}. Recall that the affine Grassmanian is given by $G(\cK)/G(\cO)$. For $G=SL_m$, a point in the affine Grassmanian can be thought of as a finitely generated, rank $m$ $\cO$-submodule of $\cK^m$ such that if $v_1, \dots, v_m$ are generators for this submodule, then $$v_1 \wedge \dots \wedge v_m=e_1 \wedge \dots \wedge e_m$$ where $e_1, \dots, e_m$ is the standard basis of $\cK^m$. Let $x_1, x_2, x_3$ be three points in the affine Grassmanian, thought of as $\cO$-submodules of $\cK^n$. We will consider the quantity $$-\val(\det(u_1, \dots, u_i, v_1, \dots v_j, w_1, \dots, w_k))$$ as $u_1, \dots, u_i$ range over elements of the $\cO$-submodule $x_1$, $v_1, \dots v_j$ range over elements of the $\cO$-submodule $x_2$, and $w_1, \dots, w_k$ range over elements of the $\cO$-submodule $x_3$. Define $f_{ijk}^t (x_1,x_2,x_3)$ as the maximum value of of this quantity, i.e., the larges value of $$-\val(\det(u_1, \dots, u_i, v_1, \dots v_j, w_1, \dots, w_k))$$ as all the vectors $u_1, \dots, u_i, v_1, \dots v_j, w_1, \dots, w_k$ range over elements of the respective $\cO$-submodules $x_1, x_2, x_3$.

There is a more invariant way to define $f_{ijk}^t$. Lift $x_1, x_2, x_3$ to elements $g_1, g_2, g_3$, of $G(\cK)$, then project to three flags $F_1, F_2, F_3 \in G/U(\cK)$. In some sense, we are lifting from $G(\cK)/G(\cO)$ to $G/U(\cK)$. Then define $f_{ijk}^t$.to be the maximum of $-\val(f_{ijk}(F_1,F_2,F_3))$ over the different possible lifts from $G(\cK)/G(\cO)$ to $G/U(\cK)$.

\begin{note} It is not hard to check that the edge functions recover the distance between two points in the affine Grassmanian (and hence also the affine building). More precisely, $f_{ij0}^t (x_1,x_2,x_3)$ is given by $\omega_j \cdot d(x_1,x_2)=\omega_i \cdot d(x_2,x_1)$.
\end{note}

\subsection{Main Definitions}

We can now give a definition of positive configurations in the affine Grassmanian.

\begin{definition} Let $x_1, x_2, \dots x_n$ be $n$ points of the affine Grassmanian. Then $x_1, x_2, \dots x_n$ will be called a positive configuration of points in the affine Grassmanian if and only if there exist ordered bases for $x_i$, 
$$v_{i1}, v_{i2}, \dots, v_{im}$$ such that for each $1 \leq p < q <r \leq n$, and each triple of non-negative integers $i, j, k$ such that $i+j+k=m$,
\begin{itemize}
\item $f_{ijk}^t (x_p,x_q,x_r) = -\val(\det(v_{p1}, \dots, v_{pi}, v_{q1}, \dots v_{qj}, v_{r1}, \dots, v_{rk}))$
\item the leading coefficient of $\det(v_{p1}, \dots, v_{pi}, v_{q1}, \dots v_{qj}, v_{r1}, \dots, v_{rk})$ is positive.
\end{itemize}
\end{definition}

This also gives us the corresponding notion of a positive configuration of points in the affine building.

\begin{definition} Let $p_1, p_2, \dots p_n$ be $n$ points of the affine building. Then $p_1, p_2, \dots p_n$ will be called a positive configuration of points in the affine building if and only if they are the image in the building of a positive configuration of points $x_1, x_2, \dots x_n$ in the affine Grassmanian.
\end{definition}

We can also define virtual configurations of points in the affine Grassmanian.

\begin{definition} A virtual positive configuration of $n$ points in the affine Grassmanian is a set of $n$ ordered pairs $(x_1,\lambda_1), \dots, (x_n,\lambda_n)$ where $x_1, \dots, x_n$ is a positive configuration of points in the affine Grassmanian, and the $\lambda_i$ are coweights.
\end{definition}

There is also a corresponding notion of a virtual positive configuration of points in the affine building.

\begin{definition} A virtual positive configuration of $n$ points in the affine building is a set of $n$ ordered pairs $(p_1,\lambda_1), \dots, (p_n,\lambda_n)$ where $p_1, \dots, p_n$ is a positive configuration of points in the affine building, and the $\lambda_i$ are coweights.
\end{definition}

Equivalence between positive configurations of $n$ points in the affine building is the smallest equivalence relation generated by isometry and cutting and gluing. Equivalence between virtual positive configurations is slightly more complicated.

If $(p_1,\lambda_1), \dots, (p_n,\lambda_n)$ is a virtual positive configuration of points in the affine building, and the coweights $\lambda_i$ are dominant, then we can associate to the virtual positive configuration $(p_1,\lambda_1), \dots, (p_n,\lambda_n)$ a positive configuration of points $p_1', \dots, p_n'$ as follows: Let $x_1, \dots x_n$ be the positive configuration of points in the affine Grassmanian giving rise to $p_1, \dots, p_n$, and let 

$$v_{i1}, v_{i2}, \dots, v_{im}$$
be a basis for $x_i$ as in the definition above, and let $x_i'$ be the lattice spanned by

$$t^{-\lambda_{i1}}v_{i1}, t^{-\lambda_{i2}}v_{i2}, \dots, t^{-\lambda_{im}}v_{im}.$$
Then $p_1', \dots, p_n'$ is defined to be the image of $x_1', \dots, x_n'$ in the affine building. We will see that $p_1', \dots, p_n'$ is again a positive configuration of points in the affine building. The definition of $p_1', \dots, p_n'$ seems to depend on the choice of the ordered bases of $x_i$, but we will see that different choices give equivalent configurations.

Now let $\lambda_i$ and $\mu_i$ be two sets of dominant coweights. Let two virtual positive configurations $(p_i, \lambda_i)$ and $(q_i, \mu_i)$ correspond to the positive configurations $p_i'$ and $q_i'$. Then we will say that $(p_i, \lambda_i)$ and $(q_i, \mu_i)$ are equivalent if and only if $p_i'$ and $q_i'$ are equivalent.

Finally, for $\lambda_i$ and $\mu_i$ any coweights, we will say that $(p_i, \lambda_i)$ and $(q_i, \mu_i)$ are equivalent if and only if for some large dominant coweights $\nu_i$, $(p_i, \lambda_i+\nu_i)$ and $(q_i, \mu_i+\nu_i)$ are equivalent.

Virtual positive configuration of $n$ points in the affine building up to equivalence are what we will define as laminations on a disc with $n$ marked points.

\begin{definition} An $\A$-lamination on a disc with $n$ marked points is a virtual positive configuration of $n$ points in the affine building up to equivalence.
\end{definition}

\subsection{Positive Configurations}

We now move on to our central construction of positive configurations of $n$ points in the affine Grassmanian, and thus positive configurations of $n$ points in the affine building. The main task involves giving a map

$$\Conf_n^+ G/U((t)) \times \Lambda^n \cdots \rightarrow \Conf_n^+ \Gr$$
which is defined for a set of large enough coweights in $\Lambda^n$. Our task will be to give criterion for ``large enough,'' and to show that when this criterion is satisfied, our map is well-defined.

For this section, let $S$ be the disc with $n$ marked points on the boundary. The associated higher Teichmuller space $\A_{G,S}$ consists of configurations of $n$ principal flags. Recall that $\cK_{>0}$ is the ring of positive Laurent series, or Laurent series with positive leading term. We will consider the set $\A_{G,S}(\cK_{>0})$. This set consists of configurations of $n$ principal flags $F_1, F_2, \dots, F_n$  in $\cK^m$. For each flag $F_i$, choose a lift to $G(\cK)=SL_m(\cK)$. Call this lift $g_i$. The columns of $g_i$ will be $v_{i1}, \dots , v_{im}$, where $v_{i1} \wedge \dots \wedge v_{ik}$ for $k=1, 2, \dots, m-1$ will be the successive subspaces (with volume form) in the flag.

One naive guess would be to associated to the flag $F_i$ the $\cO$-submodule of $\cK^m$ spanned by $v_{i1}, \dots, v_{im}$, which would then give us an element of $\Gr(G)$. This of course is not well-defined, as it would depend on our choice of lift $g_i$. However, there is a way to fix this by choosing a set of large enough coweights.

Let $\lambda$ be any coweight. For $SL_m$, a coweight is an ordered set of $m$ integers that sum to $0$. Let $\lambda=(\lambda_1,\lambda_2, \dots, \lambda_m)$. For our purposes, $\lambda$ will be large exactly when $\lambda_i-\lambda_{i+1}$ is large for $i=1, 2, \dots, m-1$. In other words, $\lambda$ is large when the pairing of $\lambda$ with each of the positive roots of $G$ is large.

For each flag $F_i$ choose a lift $v_{i1}, \dots, v_{im}$. Suppose we have some other lift $v_{i1}', \dots, v_{im}'$.

\begin{lemma}\label{independent of lift} For some large enough coweights $\lambda_i=(\lambda_{i1}, \dots, \lambda_{in})$, the vectors $$t^{-\lambda_{i1}}v_{i1}, t^{-\lambda_{i2}}v_{i2}, \dots, t^{-\lambda_{im}}v_{im}$$ generate the same $\cO$-module as $$t^{-\lambda_{i1}}v_{i1}', t^{-\lambda_{i2}}v_{i2}', \dots, t^{-\lambda_{im}}v_{im}'.$$ (Recall that a coweight for $SL_m$ consists of an $m$-tuple of integers that sum to $0$. A ``large'' coweight is one where the this $m$-tuple is decreasing, and the gaps between the integers are large.)
\end{lemma}

\begin{proof} This is fairly straightforward. We show this statement separately for each flag. Different lifts from $G/U^-(\cK)$ to $G(\cK)$ differ by some element in $U^-$. In other words, $v_{i1}, \dots, v_{im}$ and $v_{i1}', \dots, v_{im}'$ are related by some lower triangular matrix $u$ with entries in $\cK$. If the entries of $u$ and $u^{-1}$ have entries such that all their valuations are greater than $-C$ (where $C > 0$), then if we choose $\lambda_i$ large enough that its gaps $\lambda_{ij}-\lambda_{ij+1}$ are all greater than $C$, then the vectors $$t^{-\lambda_{i1}}v_{i1}, t^{-\lambda_{i2}}v_{i2}, \dots, t^{-\lambda_{im}}v_{im}$$ and the vectors $$t^{-\lambda_{i1}}v_{i1}', t^{-\lambda_{i2}}v_{i2}', \dots, t^{-\lambda_{im}}v_{im}'$$ generate the same $\cO$-module.

In more invariant terms, for such a $\lambda_i$, conjugating $u$ by $t^{\lambda_i}$ will give us an element of $G(\cO)$. This more invariant argument works for groups other than $SL_m$.

\end{proof}

The motivation for this construction can be explained as follows. A configuration of $n$ principal flags is a configuration of $n$ points of $G/U^-$ up to a diagonal action of $G$. Here by convention $U^-$ is the group of unipotent lower triangular matrices. But on this space there is a right action by $T$ on each principal flag, as $T$ is in the normalizer of $U^-$. On the space of positive configurations $\A_{G,S}(\R_{>0})$, there is an action of $T(\R_{>0})^n$. This action can be thought of as changing the horocycle (for $SL_2$, this is exactly changing the horocycle at a cusp of a hyperbolic Riemann surface). Analogously, on $\A_{G,S}(\cK_{>0})$ there is an action of $T(\cK_{>0})^n$, with one copy of $T(\cK_{>0})$ acting at each point. Inside $T(\cK_{>0})$ we have the elements $t^{\lambda}$ for $\lambda \in \Lambda$, the coweight lattice. We will take the convention that the action of $\lambda_i$ will be by multiplication by $t^{-w_0(\lambda_i)}$ on the right, which takes the flag $F_i$ given by the vectors $$v_{i1}, \dots, v_{im}$$ to the flag $$t^{-\lambda_{i1}}v_{i1}, \dots, t^{-\lambda_{im}}v_{im}.$$ Denote the resulting flag $F_i \cdot \lambda_i$. We will see that this action is the action of $\Lambda^n$ on the space of positive configurations of $n$ points in the building.

Thus although there is no sensible map from configurations in $G/U(\cK)$ to configurations in $G(\cK)/G(\cO)$, we can define a map up to some choice of lifts $v_{i1}, \dots, v_{im}$ and some choice of ``large enough'' coweights $\lambda_i$. Our strategy will then be to assign to $F_i$ the $\cO$-submodule of $\cK^m$ spanned by $$t^{-\lambda_{i1}}v_{i1}, \dots, t^{-\lambda_{im}}v_{im}.$$
Our task then becomes defining a notion of ``large enough.'' The whole construction of using a large enough coweight to associate to a point of the principal flag variety $G/U(\cK)$ a point of the affine Grassmanian $G(\cK)/G(\cO)$ can be thought of as a reverse procedure to the lifting from $G(\cK)/G(\cO)$ to $G/U(\cK)$ that we used to define the functions $f_{ijk}^t$.

Now given the configurations of $n$ principal flags $F_1, \dots, F_n$  in $\cK^m$, choose a lift $v_{i1}, \dots, v_{im}$ of each $F_i$, and choose a triangulation of the $n$-gon. We will say that the lifts $v_{i1}, \dots, v_{im}$ are {\em good} if for any triangle of flags $F_p, F_q, F_r$,

$$f_{ijk}^t (x_p,x_q,x_r)=-\val(\det(v_{p1}, \dots, v_{pi}, v_{q1}, \dots. v_{qj}, v_{r1}, \dots, v_{rk}))$$
where $x_p$ is the $\cO$-module spanned by $v_{p1},v_{p2}, \dots, v_{pm}$ and similarly for $x_q$ and $x_r$. 

In other words, the lifts $v_{i1}, \dots, v_{im}$ are good if $$-\val(f_{ijk}(F_p,F_q,F_r))=f_{ijk}^t (x_p,x_q,x_r)$$
or if they realize the maximum of minus the valuation of all the $f_{ijk}$ simultaneously. Notice that the definition of good lifts is dependent on the family of functions $f_{ijk}$. That we use these particular functions (which come from cluster algebras) will later become important, as they have certain crucial positivity properties under multiplication and addition.

Of course, it is quite unlikely to have good lifts in general. But just as before, we are saved by the the action of $\Lambda$:

\begin{lemma}\label{large enough} We can choose $\lambda_i$ large enough such that for each triangle of flags $F_p, F_q, F_r$ in our triangulation of the $n$-gon and every $i, j, k$ with $i+j+k=m$, 

\[ f_{ijk}^t (x_p,x_q,x_r) = -\val(\det(t^{-\lambda_{p1}}v_{p1},  \dots, t^{-\lambda_{pi}}v_{pi}, t^{-\lambda_{q1}}v_{q1}, \dots, t^{-\lambda_{qj}}v_{qj}, t^{-\lambda_{r1}}v_{r1}, \dots, t^{-\lambda_{rk}}v_{rk})) \] 
where now $x_p$ is the $\cO$-module spanned by $$t^{-\lambda_{p1}}v_{p1}, \dots, t^{-\lambda_{pm}}v_{pm}$$ and similarly for $x_q$ and $x_r$. In other words, we act upon the flags $F_i$, so that the vectors $$t^{-\lambda_{i1}}v_{i1}, \dots, t^{-\lambda_{im}}v_{im}$$ are a good lift of $F_i \cdot \lambda_i$ and span the $\cO$-module $x_i$. Then we will have $$-\val(f_{ijk}(F_p \cdot \lambda_p,F_q \cdot \lambda_q,F_r \cdot \lambda_r))=f_{ijk}^t (x_p,x_q,x_r).$$
\end{lemma}

\begin{proof} The different spans $x_p, x_q, x_r$ vary as we change the $\lambda_i$. Consider all the different possible values for minus the valuation of the determinant of some subset of $i$ vectors among $$t^{-\lambda_{p1}}v_{p1}, \dots, t^{-\lambda_{pm}}v_{pm},$$ some subset of $j$ vectors among $$t^{-\lambda_{q1}}v_{q1}, \dots, t^{-\lambda_{qm}}v_{qm},$$ and some subset of $k$ vectors among $$t^{-\lambda_{r1}}v_{r1}, \dots, t^{-\lambda_{rm}}v_{rm}.$$ Observe that for any choice of $\lambda_i$, $f_{ijk}^t (x_p,x_q,x_r)$ is the maximum of all these values, as the determinant of any $i$ vectors in $x_p$, $j$ vectors in $x_q$, and $k$ vectors in $x_r$ is a linear combination of the determinants considered above.

However, as the $\lambda_i$ get large, $\lambda_{ij}-\lambda_{i {j+1}}$ all get large simultaneously, so that the valuation of 
$$\det(t^{-\lambda_{p1}}v_{p1}, \dots, t^{-\lambda_{pi}}v_{pi}, t^{-\lambda_{q1}}v_{q1}, \dots, t^{-\lambda_{qj}}v_{qj}, t^{-\lambda_{r1}}v_{r1}, \dots, t^{-\lambda_{rk}}v_{rk})$$ gets negative the fastest among determinants of sets of $i$, $j$ and $k$ vectors from $x_p, x_q$ and $x_r$, respectively. Thus for large enough $\lambda_i$, $$-\val(\det(t^{-\lambda_{p1}}v_{p1}, \dots, t^{-\lambda_{pi}}v_{pi}, t^{-\lambda_{q1}}v_{q1}, \dots, t^{-\lambda_{qj}}v_{qj}, t^{-\lambda_{r1}}v_{r1}, \dots, t^{-\lambda_{rk}}v_{rk}))$$ will be the largest value among the different negative valuations of the determinants. Thus this value will indeed be $f_{ijk}^t (x_p,x_q,x_r)$, and we have our claim.

\end{proof}

We observe here that if $F_i$ is a positive configuration of flags in $G/U((t))$, then $F_i \cdot \lambda_i$ will also be a positive configuration of flags for any choice of $\lambda_i$. Now we can define positive configurations in the affine Grassmanian. Given a positive configuration of flags $F_1, F_2, \dots, F_n$ coming from a point in $\A_{G,S}(\cK_{>0})$, we choose some lifts $v_{i1}, \dots, v_{im}$ of the $F_i$ and choose a triangulation of the $n$-gon. Then taking $\lambda_i$ large enough to give us good lifts as in Lemma~\ref{large enough}, we can then obtain the points $x_1, x_2, \dots, x_n \in G(\cK)/G(\cO)$ also as above. We will call configurations of points in the affine Grassmanian that arise in this way {\em positive configurations of points}.

\subsubsection{Dependence on choices} We now must analyze how this construction depended on various choices, for example, the choice of lifts, the choice of triangulations or the choice of $\lambda_i$. We will consider these each in turn.

It turns out that the choice of lifts $v_{i1}, \dots, v_{im}$ only affects the choice of $\lambda_i$. This will be our next lemma. First, note that if $t^{-\lambda_{i1}}v_{i1}, \dots, t^{-\lambda_{im}}v_{im}$ is a good lift of $F_i \cdot \lambda_i$, then replacing the $\lambda_i$ with any set of larger coweights will still give us a good lift.

\begin{lemma}\label{independence of lift} If we have two different sets of lifts $v_{i1},\dots, v_{im}$ and $v_{i1}', \dots, v_{im}'$, and some $\lambda_i$ such that $t^{-\lambda_{i1}}v_{i1}, \dots, t^{-\lambda_{im}}v_{im}$ and $t^{-\lambda_{i1}}v_{i1}', \dots, t^{-\lambda_{im}}v_{im}'$ are both good lifts of $F_i \cdot \lambda_i$, then in fact both sets of vectors span the same $\cO$-modules.
\end{lemma}
\begin{proof} The two sets of vectors $$t^{-\lambda_{i1}}v_{i1}, \dots, t^{-\lambda_{im}}v_{im}$$ and $$t^{-\lambda_{i1}}v_{i1}', \dots, t^{-\lambda_{im}}v_{im}'$$ differ by lower triangular matrices $U_i$ which take the former to the latter. The entries of these matrices must be in $\cO$, otherwise the sets of vectors could not both be good lifts.

For example, suppose $j < k$ and $t^{-\lambda_{1k}}v_{1k}'=t^{-\lambda_{1k}}v_{1k} + a \cdot t^{-\lambda_{1j}}v_{1j}$, where $a$ has negative valuation. Then if $t^{-\lambda_{i1}}v_{i1},\dots, t^{-\lambda_{im}}v_{im}$ is a good lift, then replacing any occurence of $t^{-\lambda_{1j}}v_{1j}'$ by $t^{-\lambda_{1k}}v_{1k}'$ in some determinant expression will result in a smaller valuation, so that $t^{-\lambda_{i1}}v_{i1}', \dots, t^{-\lambda_{im}}v_{im}'$ won't be a good lift.
\end{proof}

We now analyze the dependence on the triangulation. We have the following:

\begin{lemma}\label{independent of triangulation} If $F_i \cdot \lambda_i$ comes from a positive configuration and has good lifts $$t^{-\lambda_{i1}}v_{i1}, \dots, t^{-\lambda_{im}}v_{im}$$ for some triangulation of the $n$-gon, this lift remains good for any other triangulation. In other words, if we change the triangulation of the $n$-gon, we do not need to change the $\lambda_i$.
\end{lemma}

\begin{proof} The general case is equivalent to the case where $\lambda_i=0$, so we assume this for simplicity. Every change of triangulation comes from a sequence of flips; thus is suffices to consider the case where we have just four flags $F_i$, $i=1, 2, 3, 4$. We will consider this simpler case.

Assume that $v_{i1}, \dots,v_{im}$ are a good lift of $F_i$ in the triangulation with triangles $123, 134$. We want to show that they remain good lifts for the triangles $124, 234$. Let the span of $v_{i1}, \dots, v_{im}$ be $x_i$. 

We already know that 

$$f_{ijk}^t (x_p,x_q,x_r)=-\val(\det(v_{p1}, \dots,v_{pi},v_{q1}, \dots,v_{qj}, v_{r1}, \dots, v_{rk}))$$
for $(p,q,r)=(1,2,3)$ and $(1,3,4)$, and want to conclude this for $(p,q,r)=(1,2,4)$ and $(2,3,4)$. Recall that $f_{ijk}^t (x_p,x_q,x_r)$ is defined by taking the maximum of minus the valuation of the determinant of some subset of $i$ vectors in $x_p$, some subset of $j$ vectors in $x_q$, and some subset of $k$ vectors in $x_r$. Equivalently, it comes from taking the maximum of minus the valuation of the determinant of some $i$-dimensional subspace of the $\R$ span of $v_{p1}, \dots, v_{pm}$, some $j$-dimensional subspace of  $\R$ span of $v_{q1}, \dots,v_{qm}$, and some $k$-dimensional subspace of the $\R$ span of $v_{r1}, \dots, v_{rm}$. Among all the possile choices for these subspaces (the space of choices is parameterized by a product of Grassmanians), the set of them that achieve the maximum of minus the valuation of the determinant is an open set.

Now consider the space of flags in the $\R$ span of $v_{i1}, \dots,v_{im}$. Let $y_i$ for $i=1, 2, 3, 4$ be any flags in the span of $v_{i1}, \dots,v_{im}$, and let them be represented by $w_{i1}, \dots, w_{im}$, where the $r$-dimensional subspace in $y_i$ is the span of $w_{i1}, w_{i2}, \dots, w_{ir}$. Then there is an open subset of choices for the $y_i$ such that 

$$f_{ijk}^t (x_p,x_q,x_r)=-\val(\det(w_{p1}, \dots, w_{pi}, w_{q1}, \dots. w_{qj}, w_{r1}, \dots, w_{rk}))$$
for each triples $(p,q,r)=(1,2,3), (1,3,4), (1,2,4)$ and $(2,3,4)$ simultaneously.

Thus there is some choice of $y_i$ such that the above equality holds for all $i, j, k$ as well as for all triples $(p,q,r)$. Now let $$f'_{ijk}(p,q,r)=\det(w_{p1}, \dots, w_{pi}, w_{q1}, \dots. w_{qj}, w_{r1}, \dots, w_{rk}).$$ Then we have $$f_{ijk}^t (x_p,x_q,x_r)=-\val(f'_{ijk}(p,q,r)).$$

We know that $$-\val(f_{ijk}(F_p,F_q,F_r))=-\val(f'_{ijk}(p,q,r))=f_{ijk}^t (x_p,x_q,x_r)$$ for $(p,q,r)=(1,2,3)$ and $(1,3,4).$ We would like to conclude this also for $(p,q,r)=(1,2,4)$ and $(2,3,4)$.

By \cite{FG1}, we know that all the functions $$f'_{ijk}(1,2,4), \textrm{ }  f'_{ijk}(2,3,4)$$ can be expressed in terms of the functions $$f'_{ijk}(1,2,3), \textrm{ } f'_{ijk}(1,3,4)$$ using only addition, multiplication and division. This gives us that $$-\val(f'_{ijk}(1,2,4)) \textrm{, } -\val(f'_{ijk}(2,3,4))$$ are less than or equal to some tropical expression in $$-\val(f'_{ijk}(1,2,3)) \textrm{, } -\val(f'_{ijk}(1,3,4)).$$

However, by positivity, we know that $$-\val(f_{ijk}(F_1,F_2,F_4)) \textrm{, } -\val(f_{ijk}(F_2,F_3,F_4))$$ are equal to (and are not just less than or equal to) the same tropical expressions in $$-\val(f_{ijk}(F_1,F_2,F_3)) \textrm{, } -\val(f_{ijk}(F_1,F_3,F_4)).$$

Therefore we have that $$-\val(f'_{ijk}(1,2,4)) \leq -\val(f_{ijk}(F_1,F_2,F_4))$$ and $$-\val(f'_{ijk}(2,3,4)) \leq -\val(f_{ijk}(F_2,F_3,F_4)).$$ Thus by the maximality of $$-\val(f'_{ijk}(1,2,4)) \textrm{, } -\val(f'_{ijk}(2,3,4)),$$ we must have $$-\val(f'_{ijk}(1,2,4))=-\val(f_{ijk}(F_1,F_2,F_4))$$ and $$-\val(f'_{ijk}(2,3,4))=-\val(f_{ijk}(F_2,F_3,F_4)).$$ Therefore 

$$f_{ijk}^t (x_p,x_q,x_r)=-\val(f'_{ijk}(p,q,r))=-\val(f_{ijk}(F_p,F_q,F_r))$$ $$=-\val(\det(v_{p1}, \dots,v_{pi},v_{q1}, \dots,v_{qj},v_{r1}, \dots, v_{rk}))$$
for $(p,q,r)=(1,2,4)$ and $(2,3,4)$.

\end{proof}

Thus the map from configurations in $G/U(\cK)$ to configurations in $G(\cK)/G(\cO)$ does not depend on the choice of triangulation. Of course, the map does depend on choosing the $\lambda_i$: we get a different map for each choice of $\lambda_i$, or a map

$$\A_{G,S}(\cK_{>0}) \times \Lambda^n = \Conf_n^+ G/U((t)) \times \Lambda^n \cdots \rightarrow \Conf_n^+ \Gr$$
Thus it is better to think that for every point in $\A_{G,S}(\cK_{>0})$, we have a cone of positive configurations in $G(\cK)/G(\cO)$, parameterized by all choices of $\lambda_i$ that are ``large enough'' for some lifts of the flags. These families of positive configurations in $\Gr$ give a configuration of cones inside $\Gr$.

It turns out that there exist lifts $g_i$ of $F_i$ such that we may choose the $\lambda_i$ minimally. In other words, a set of coweights $\lambda_i$ is large enough for some choice lifts only if they are large enough for the particular lifts $g_i$. A statement equivalent to this, though in somewhat different language, appears in \cite{GS}. This is discussed further in section 5.7.

\subsection{Virtual Positive Configurations}

To give a better framework for thinking about these cones of configurations of points, we will define virtual positive configurations. A virtual positive configuration of $n$ points in the affine Grassmanian consists of $n$ pairs $(x_1,\lambda_1), \dots, (x_n,\lambda_n)$ where $x_1, \dots, x_n$ is a positive configuration of points in the affine Grassmanian, and the $\lambda_i$ are coweights.

Starting with a positive configuration of $n$ flags $F_1, \dots, F_n$, let  $v_{i1}, \dots, v_{im}$ be some lifts of these flags. Then suppose that for a choice of $\lambda_i$, $t^{-\lambda_{i1}}v_{i1}, \dots, t^{-\lambda_{im}}v_{im}$ is a good lift of $F_i \cdot \lambda_i$. Let $x_i$ be the $\cO$-module spanned by $t^{-\lambda_{i1}}v_{i1}, \dots, t^{-\lambda_{im}}v_{im}$. Then we can associate to $F_1, \dots, F_n$ the virtual positive configuration $$(x_1,-\lambda_1), \dots, (x_n,-\lambda_n)$$

Observe that by Lemma~\ref{independence of lift} the $x_i$ only depend on the $\lambda_i$ and not the lifts $v_{i1}, \dots v_{im}$. If instead we had chosen a different set of coweights $\lambda'_i$ instead of $\lambda_i$ (for example, it is clear that we could have taken any set of $\lambda'_i$ such that each was larger than the corresponding $\lambda_i$), we would have ended up with a different virtual positive configuration $(x'_1,-\lambda'_1), \dots, (x'_n,-\lambda'_n)$. Thus we get many different virtual positive configurations $(x_1,-\lambda_1), \dots, (x_n,-\lambda_n)$ as the set of $\lambda_i$ varies over some subset of $\Lambda^n$. All we know about this subset is that it is closed under addition by any element of the monoid $\Lambda_+^n$. (It is shown in \cite{GS} that there is a smallest possible value for the $\lambda_i$ such that all other values can be obtained by adding some element of the moniod $\Lambda_+^n$--see the last paragraph of the previous section.)

We would like to consider these different virtual configurations that arise for different choices of $\lambda_i$ to be equivalent. 

\begin{definition} A {\em family of virtual positive configurations of points in $\Gr(G)$} is a set of virtual positive configurations $(x_1,-\lambda_1)$, $\dots, (x_n,-\lambda_n)$ (the $x_i$ vary with choices of $\lambda_i$), which comes from a single positive configuration of flags $F_i$ in $G/U^-(\cK)$.
\end{definition}

Thus we have a well-defined map from a positive configuration in $G/U(\cK)$ to a family of virtual configurations in $G(\cK)/G(\cO)$. One can check as an exercise (and though we do not logically need this it is useful psychologically), that knowing only such a family of virtual configurations in $G(\cK)/G(\cO)$, we can reconstruct the original positive configuration in $G/U(\cK)$ up to the action of $T(\R)$ on each flag. This is in contrast with the fact that any particular virtual configuration can come from many different positive configurations in $G/U(\cK)$.

We can now define an action of of $\Lambda^n$ on virtual positive configurations of points in $\Gr(G)$. (This is different than how $\Lambda^n$ parameterizes different virtual positive configurations in a family of virtual positive configurations).

Note that if $(x_1,-\lambda_1)$, $\dots, (x_n,-\lambda_n)$ is a virtual positive configuration, then so is $(x_1,-\lambda_1 + \mu_1)$, $\dots$, $(x_n,-\lambda_n + \mu_n)$ for any coweights $\mu_1, \dots, \mu_n$. Whereas $(x_1,-\lambda_1), \dots, (x_n,-\lambda_n)$ came from a family of virtual positive configurations associated to $F_1, \dots, F_n$, the new configuration comes from a family of virtual positive configuration of flags $F_1 \cdot \mu_1, \dots, F_n \cdot \mu_n$: 
$$F_1, \dots, F_n \rightarrow (x_1,-\lambda_1), \dots, (x_n,-\lambda_n)$$
$$F_1 \cdot \mu_1, \dots, F_n \cdot \mu_n \rightarrow (x_1,-\lambda_1 + \mu_1), \dots, (x_n,-\lambda_n + \mu_n).$$

We define the action of $\Lambda^n$ on virtual positive configurations in this way: it takes the virtual configuration of points $(x_1,-\lambda_1)$, $\dots, (x_n,-\lambda_n)$ in $\Gr(G)$ to the virtual configuration $(x_1,-\lambda_1 + \mu_1)$, $\dots, (x_n,-\lambda_n + \mu_n)$. Note that this action takes virtual configurations in one family to virtual configurations in another; in other words, it respects families of virtual configurations of points.

We will say that a virtual positive configuration $(x_1,-\lambda_1), \dots, (x_n,-\lambda_n)$ is realized by an actual configuration $x'_1, \dots, x'_n$ if $(x_1,-\lambda_1), \dots, (x_n,-\lambda_n)$ and $(x'_1,0), \dots, (x'_n,0)$ lie in some family of virtual positive configurations. Similarly, we make the analogous definition for virtual positive configurations on points in the affine building: a virtual positive configuration $(p_1,-\lambda_1), \dots, (p_n,-\lambda_n)$ is realized by an actual configuration $p'_1, \dots, p'_n$ if $(p_1,-\lambda_1), \dots, (p_n,-\lambda_n)$ and $(p'_1,0), \dots, (p'_n,0)$ lie in some family of virtual positive configurations

Let us now make some observations. Suppose we have a positive configuration of $n$ flags $F_1, \dots, F_n$, and let $(x_1,-\lambda_1)$, $\dots,$ $(x_n,-\lambda_n)$ be any of the associated virtual configurations. By definition, the positive configuration $x_1, \dots, x_n$ is defined by good lifts, so that we know that 

$$-\val(f_{ijk}(F_p \cdot \lambda_p,F_q \cdot \lambda_q,F_r \cdot \lambda_r))=f_{ijk}^t (x_p,x_q,x_r)$$
for all $1 \leq p, q, r \leq n$ and $0 \leq i, j, k\leq m-1$ with $i+j+k=m$. It is also clear from direct calculation that $$f_{ijk}(F_p \cdot \lambda_p,F_q \cdot \lambda_q,F_r \cdot \lambda_r) = f_{ijk}(F_p, F_q, F_r) \cdot t^{-\lambda_p \cdot \omega_i -\lambda_q \cdot \omega_j -\lambda_r \cdot \omega_k}$$ where $\omega_i, \omega_j, \omega_k$ are fundamental weights, so that we can conclude that 

$$f_{ijk}^t (x_p,x_q,x_r)=-\val(f_{ijk}(F_p, F_q, F_r)) + \lambda_p \cdot \omega_i + \lambda_q \cdot \omega_j + \lambda_r \cdot \omega_k.$$

We extend the definition of $f_{ijk}^t$ to virtual positive configurations of points in the affine Grassmanian as follows. We define $$f_{ijk}^t ((x_p, \lambda_p),(x_q, \lambda_q),(x_r, \lambda_r)) = f_{ijk}^t (x_p,x_q,x_r) + \lambda_p \cdot \omega_i + \lambda_q \cdot \omega_j +\lambda_r \cdot \omega_k.$$

Then if the positive configuration of flags $F_i$ is associated with the virtual positive configuration $(x_i,\lambda_i)$, then we have $$f_{ijk}^t ((x_p,\lambda_p),(x_q,\lambda_q),(x_r,\lambda_r)) = -\val(f_{ijk} (F_p,F_q,F_r)),$$
and this is true for any virtual positive configuration $(x_i,\lambda_i)$ in the family of virtual positive configurations associated to the positive configuration of flags $F_i$.

\subsection{Definition of laminations for the $\A$ space of a disc with marked points}

We are now very close to getting a complete definition of higher laminations. We have defined the functions $f_{ijk}^t$ on virtual positive configurations in the affine Grassmanian. Moreover, because these virtual positive configurations come from positive configurations of flags we have the following:

\begin{prop}\label{tropical relations} Suppose we have a virtual positive configuration of points in the affine Grassmanian $(x_i,\lambda_i)$. Then all the functions $f_{ijk}^t ((x_p, \lambda_p),(x_q, \lambda_q),(x_r, \lambda_r))$ for different $i$, $j$, $k$, $p$, $q$, $r$ satisfy the tropical relations satisfied by tropical points of $\A_{G,S}$.

\end{prop}

\begin{proof} The functions $f_{ijk} (F_p,F_q,F_r)$ satisfy some relations defining $\A_{G,S}$. These relations involve only addition, multiplication, and division. Moreover, because we have a positive configuration, $$f_{ijk} (F_p,F_q,F_r) \in \cK_{>0}.$$ Therefore the negative valuations of these functions
$$-\val(f_{ijk} (F_p,F_q,F_r))$$
must satisfy the tropicalizations of the corresponding relations. Because $$f_{ijk}^t ((x_p,\lambda_p),(x_q\lambda_q),(x_r\lambda_r)) = -\val(f_{ijk} (F_p,F_q,F_r)),$$ the functions $f_{ijk}^t ((x_p,\lambda_p),(x_q,\lambda_q),(x_r,\lambda_r))$ satisfy the tropical relations defining $\A_{G,S}(\Z^t)$ as well.

\end{proof}

We can thus associate to any tropical point of $\A_{G,S}$ some virtual positive configuration of points in the affine Grassmanian. The problem is that there are many associated virtual positive configurations: not only do they come in families, but there are many points in $\A_{G,S}(\cK)$ which have identical valuations.

This is where it becomes useful to look at virtual positive configurations of points in the affine building. Let $(x_1,-\lambda_1)$, $\dots,$ $(x_n,-\lambda_n)$ be a virtual positive configurations in the affine Grassmanian, and suppose the points $x_1, \dots, x_n$ map to points $p_1, \dots, p_n$ in the affine building. In this situation, we will say that $$(p_1,-\lambda_1), \dots, (p_n,-\lambda_n)$$ is the corresponding virtual positive configuration of points in the affine building. We can similarly define families of virtual positive configurations of points in $\Delta(G)$ coming from a positive configuration of flags $F_i \in G/U((t))$. We extend the tropical functions to the virtual positive configurations of points in the affine building as follows:

$$f_{ijk}^t ((p_p,\lambda_p),(p_q,\lambda_q),(p_r,\lambda_r))=f_{ijk}^t ((x_p,\lambda_p),(x_q,\lambda_q),(x_r,\lambda_r)).$$

The reason we want to view these configurations inside the affine building rather than the affine Grassmanian is that the affine building is where it is most natural to define equivalence between configurations of points. The finer algebro-geometric properties of the affine Grassmanian are not reflected in the tropical functions; it is precisely the metric structure of the building which is captured in the tropical functions.

Our goal will be to show that tropical points $\A_{G,S}(\Zt)$ correspond precisely to virtual positive configurations of points in the affine building up to equivalence. Thus we seek some notion of equivalence such that all virtual positive configurations of points in a given family will be equivalent. We now give such a definition.

We will repeat here some of the treatment from section 5.1 defining virtual positive configurations of points in the affine building. Recall that equivalence of two positive configurations of points in the affine building is the minimal equivalence relation generated by isometry and cutting and gluing.

Let $p_1, \dots, p_n$ and $p'_1, \dots, p'_n$ be two positive configurations of points of the affine building. The first requirement for equivalence of these positive configurations is that the pairwise distances between corresponding points are equal:

$$d(p_i,p_j)=d(p'_i,p'_j)$$

The {\em perimeter} of a configuration $p_1, \dots, p_n$ is a union of some choice of geodesics between each $p_i$ and $p_{i+1}$, where indices are taken cyclically. Then, because
$$d(p_i,p_{i+1})=d(p'_i,p'_{i+1}),$$
the set of geodesics between $p_i$ and $p_{i+1}$ and the set of geodesics between $p_i$ and $p_{i+1}$ can be canonically identified. Thus we may choose a perimeter for $p_1, \dots, p_n$ and the corresponding perimeter for $p'_1, \dots, p'_n$.

Now let $a$ and $b$ be two points in the perimeter of the first configuration. Suppose that $a$ is on the between $p_i$ and $p_{i+1}$ and $b$ is on the geodesic between $p_j$ and $p_{j+1}$. We may choose some geodesic between $a$ and $b$. Take the corresponding points $a'$ and $b'$ on the perimeter of $p'_1, \dots, p'_n$. Then we make a ``cut'' to form the configurations $$a, p_{i+1}, \dots, p_j, b$$ and $$b, p_{j+1}, \dots, p_i, a$$ and the corresponding configurations $$a', p'_{i+1}, \dots, p'_j, b'$$ and $$b', p'_{j+1}, \dots, p'_i, a'.$$ By Lemma~\ref{adding perimeter points} below, we will see that all these resulting configurations are still positive configurations of points in the affine building.

Then we will say that $p_1, \dots, p_n$ and $p'_1, \dots, p'_n$ are equivalent if and only if $a, p_{i+1}, \dots, p_j, b$ is equivalent to $a', p'_{i+1}, \dots, p'_j, b'$ and $b, p_{j+1}, \dots, p_i, a$ is equivalent to $b', p'_{j+1}, \dots, p'_i, a'$. Using cuts, we can reduce to the case of triangles with miniscule side lengths. Finally, we say that two such triangles are equivalent if their side lengths coincide, which in turn implies that they are isometric. Note that for larger triangles, with not necessarily miniscule side lengths, it is not generally the case the the side lengths determine the triangle up to isometry. This is an unusual property of distances in affine buildings. Our definition of equivalence of equivalence does not depend on the sequence of cuts, as we will see in the main theorem below Theorem~\ref{tropical points}.

Let us now extend this equivalence relation to {\em virtual} positive configurations of points in the affine building.

Let $(p_i, \lambda_i)$ and $(q_i, \mu_i)$ be two virtual configurations. Suppose that all the $\lambda_i$ and $\mu_i$ are positive coweights. Then both $(p_i, \lambda_i)$ and $(q_i, \mu_i)$ come in a family containing the actual configurations $(p'_i, 0)$ and $(q'_i, 0)$. We will say that the positive configuration $p_1', \dots , p_n'$ realizes the virtual configuration $(p_i, \lambda_i)$. Then we will say that $(p_i, \lambda_i)$ and $(q_i, \mu_i)$ are equivalent virtual configurations if and only if $p'_1, \dots, p'_n$ and $q'_1, \dots, q'_n$ are equivalent as configurations. 

More generally, two configurations $(p_i, \lambda_i)$ and $(q_i, \mu_i)$ are equivalent if there exists $\nu_i$ such that $(p_i, \lambda_i+\nu_i)$ and $(q_i, \mu_i+\nu_i)$ are equivalent. For large enough $\nu_i$, $\lambda_i+\nu_i$ and $\mu_i+\nu_i$ will be dominant, so that $(p_i, \lambda_i+\nu_i)$ and $(q_i, \mu_i+\nu_i)$ can be realized by some configurations $p'_1, \dots, p'_n$ and $q'_1, \dots, q'_n$. In other words, if $(p_i, \lambda_i)$ and $(q_i, \mu_i)$ are equivalent, they are in the same families as $(p'_i, -\nu_i)$ and $(q'_i, -\nu_i)$ for some $\nu_i$, respectively, and where $p'_1, \dots, p'_n$ and $q'_1, \dots, q'_n$ are equivalent positive configurations.

From this definition it is clear that virtual positive configurations coming from the same family are equivalent. Thus every family of virtual positive configurations of points in the affine Grassmanian
$$\Conf_n^{+,\textrm{vir}} \Gr$$
gives, up to equivalence, a single virtual configuration of points in the affine building, which we will denote
$$\Conf_n^{+,\textrm{vir}} \Delta(G).$$

Given a virtual positive configuration $(p_i, \lambda_i)$ of points in the affine building and a trianglulation of the $n$-gon, we associate to each triangle $(p_a,\lambda_a), (p_b\lambda_b), (p_c\lambda_c)$ the tropical functions

$$f_{ijk}^t ((p_a,\lambda_a),(p_b\lambda_b),(p_c\lambda_c)) :=f_{ijk}^t ((x_a,\lambda_a),(x_b\lambda_b),(x_c\lambda_c)) = \val(f_{ijk} (F_a,F_b,F_c)).$$

Here is the main theorem of this paper

\begin{theorem}\label{tropical points} There is a bijection between tropical points of $\A_{G,S}(\Zt)$ and virtual positive configurations of points in the affine building up to equivalence. Given a virtual positive configuration $(p_i, \lambda_i)$ of points in the affine building, it comes from a virtual positive configuration of points $(x_i, \lambda_i)$ in the affine Grassmanian, which in turn comes from a positive configuration of flags $F_i \in G/U^-(\cK)$. Given a trianglulation of the $n$-gon, we associate to each triangle $(p_a,\lambda_a), (p_b\lambda_b), (p_c\lambda_c)$ the functions

$$f_{ijk}^t ((p_a,\lambda_a),(p_b\lambda_b),(p_c\lambda_c)) :=f_{ijk}^t ((x_a,\lambda_a),(x_b\lambda_b),(x_c\lambda_c)) = -\val(f_{ijk} (F_a,F_b,F_c)).$$

These functions satisfy the tropical relations, and therefore give a well-defined point of 
$\A_{G,S}(\Zt)$. The values of these functions completely determine the virtual positive configurations of points in the affine building up to equivalence, and vice versa.
\end{theorem}

\begin{proof} Most of this theorem was already proved in Proposition~\ref{tropical relations}. We need only to show that two virtual positive configurations of points in the affine building have the same tropical coordinates $f_{ijk}^t$ if and only if they are equivalent.

Recall that any two virtual positive configurations that lie in the same family have the same tropical coordinates $f_{ijk}^t$. Moreover, it is clear from the definitions that virtual positive configurations in the same family are equivalent. Also note that the properties of equivalence and having the same coordinates are stable under the action of $\Lambda$: two virtual positive configurations of points in the affine building $(p_i, \lambda_i)$ and $(q_i, \mu_i)$ have the same coordinates if and only if for any choice of coweights $\lambda_i$, $(p_i, \lambda_i+\nu_i)$ and $(q_i, \mu_i+\nu_i)$ have the same coordinates. Thus it suffices to work with positive configurations $p_i$ and $q_i$ and show that they are equivalent if and only if they have the same (tropical) coordinates.

Let $p_i$ and $q_i$ be two positive configurations of points in the affine building. Let them come from positive configurations $x_i$ and $x'_i$ in the affine Grassmanian which in turn come from positive configurations $F_i$, $F'_i$ of flags in $G/U^-(\cK)$. The following technical lemma will be the crux of our proof, and we will make repeated use of it:

\begin{lemma}\label{adding perimeter points} Given any positive configuration $p_1, \dots, p_n$ in the affine building, let $y$ be a point on some geodesic between $p_a$ and $p_{a+1}$. Then $p_1, \dots, p_a, y, p_{a+1}, \dots, p_n$ is a positive configuration of $n+1$ points in the affine building.
\end{lemma}

\begin{proof} We can actually do this on the level of flags. The idea is to construct a flag $F$ such that $F_1$, $\dots,$ $F_a$, $F$, $F_{a+1}, \dots,$ $F_n$ is a positive configuration of flags that maps down to $p_1$, $\dots,$ $p_a$, $y$, $p_{a+1}, \dots,$ $p_n$ in the affine building.

Let us denote by $$f_{i,m-i}(F_a,F_{a+1})$$ the edge function corresponding to $$\det(v_1, \dots, v_i, w_1, \dots, w_{m-i})$$ where the sequence of vectors $v_1, \dots, v_m$ gives the flag $F_a$ and the sequence of vectors $w_1, \dots, w_m$ gives the flag $F_{a+1}$. Let $f_{i,m-i}^t$ be the corresponding tropical function given by the negative valuation of $f_{i,m-i}$. Recall that $$f_{i,m-i}^t (p_a,p_{a+1})=d(p_a,p_{a+1}) \cdot \omega_{m-i}.$$

Now we want to construct a flag $F$ that will map down to the point $y$ in the affine building. Note that $$d(p_a,y)+d(y,p_{a+1})=d(p_a,p_{a+1}).$$ Let $$d_i = d(p_a,y) \cdot \omega_{m-i}.$$ We will construct $F$ by stipulating that $$f_{ij0}(F_a, F, F_{a+1})=t^{d_i},$$ $$f_{0jk}(F_a, F, F_{a+1})=f_{j,k}(F_a,F_{a+1})t^{-d_j},$$ $$f_{ijk}(F_a, F, F_{a+1})=f_{i+j,k}(F_a,F_{a+1})t^{-d_{i+j}+d_i}.$$

One sees that by construction, $F_a,  F, F_{a+1}$ is a postive configuration in $G/U^-(\cK)$. Moreover, it maps to an actual configuration of points in the affine Grassmanian, not merely a virtual one. Consider the triple of flags 
$$F_a \cdot -d(y,p_a), F, F_{a+1} \cdot -d(y,p_{a+1}).$$ If we show that this gives an actual configuration in the building, then because $d(y,p_a)$ and $d(y,p_{a+1})$ are positive coweights, $p_a,  y, p_{a+1}$ will also be an actual configuration.

But one can calculate that the functions 
$$f_{ijk}(F_a \cdot -d(y,p_a), F, F_{a+1} \cdot -d(y,p_{a+1}))$$ all have valuation $0$; $F$ was essentially constructed so that this would be the case. This means that that we may choose the flags $F_a \cdot -d(y,p_a), F, F_{a+1} \cdot -d(y,p_{a+1})$ to be generated by vectors of valuation $0$: an explicit formula is given in section 9 of \cite{FG1} for the sets of vectors generating three flags in terms of the functions $f_{ijk}$. These three sets of vectors will span the standard $\cO$-module consisting of all vectors with valuation greater than or equal to $0$. Thus, the flags $F_a \cdot -d(y,p_a), F, F_{a+1} \cdot -d(y,p_{a+1})$ are associated to points $\tilde{x}_a, x, \tilde{x}_{a+1}$ in the affine Grassmanian which all coincide.

Then $f_{ijk}^t(x_a, x, x_{a+1})$ is exactly $0$, so 
$$-\val(f_{ijk}(F_a \cdot -d(y,p_a), F, F_{a+1} \cdot -d(y,p_{a+1})))=f_{ijk}^t(\tilde{x}_a, x, \tilde{x}_{a+1}),$$ and therefore the vectors that generated $\tilde{x}_a, x, \tilde{x}_{a+1}$ gave a good lift of
$$F_a \cdot -d(y,p_a), F, F_{a+1} \cdot -d(y,p_{a+1}).$$ Thus $\tilde{x}_a, x, \tilde{x}_{a+1}$ is a positive configuration of points in the affine Grassmanian. Therefore, because $d(y,p_a)$ and $d(y,p_{a+1})$ are positive coweights, $$F_a, F, F_{a+1}$$ will give a positive configuration of points $x_a, x, x_{a+1}$ in the affine Grassmanian, which maps down to a positive configuration of points $p_a, y, p_{a+1}$ in the affine building. By Lemma~\ref{independent of triangulation} we may glue to see that $$p_1, \dots, p_a, y, p_{a+1}, \dots, p_n$$ is a positive configuration of points in the building.

\end{proof}

We now return to a proof the theorem. Suppose we have two positive configurations of points $p_1, p_2, \dots, p_n$ and $q_1, q_2, \dots, q_n$ that under some (and hence any) triangulation of an $n$-gon have the same coordinates. We wish to show that they are equivalent.

First observe that the distance between any two points $p_i$ and $p_j$ is the same as the distance between the corresponding points $q_i$ and $q_j$ by virtue of the coordinates of the configurations being the same--recall that the edge functions give the distances between points in a configuration. In particular, the distance between $p_i$ and $p_{i+1}$ is the same as the distance between $q_i$ and $q_{i+1}$. Then choosing any geodesic between $p_i$ and $p_{i+1}$, we can choose the corresponding geodesic between $q_i$ and $q_{i+1}$, and if we like, we may extend the positive configurations $p_1, p_2, \dots, p_n$ and $q_1, q_2, \dots, q_n$ by adding points along the geodesics to form a perimeter. Assume we had done this to begin with, and that the result was that we had two configurations of points in the affine building $p_1, p_2, \dots, p_n$ and $q_1, q_2, \dots, q_n$ that were both perimeters.

Now take any two points $p_i$ and $p_j$. Then take any geodesic between them, and take the corresponding geodesic between $q_i$ and $q_j$. Then $p_i, p_{i+1}, \dots, p_j$ and $p_j, p_{j+1}, \dots, p_i$ are both positive configurations of points. Moreover, their coordinates are completely determined by the configuration $p_1, p_2, \dots, p_n$. Therefore, the corresponding configurations $q_i, q_{i+1}, \dots, q_j$ and $q_j, q_{j+1}, \dots, q_i$ are both positive and have the same respective coordinates as the configurations $p_i, p_{i+1}, \dots, p_j$ and $p_j, p_{j+1}, \dots, p_i$.

Thus the property of having the same coordinates is stable under the operation of cutting. Continuing this process, we may repeatedly cut the configurations $p_1, p_2, \dots, p_n$ and $q_1, q_2, \dots, q_n$ until we get miniscule triangles which are determined up to isometry by their side lengths, and therefore by the functions $f_{ijk}^t$.

Now let us show the inverse statement. Suppose that $p_1, p_2, \dots, p_n$ and $q_1, q_2, \dots, q_n$ don't have the same coordinates. If the pairwise distances between each $p_i$ and $p_j$ and the corresponding $q_i$ and $q_j$ aren't the same, then clearly they can't be equivalent. If they are the same, then let us make some ``cut''� of the configuration choosing any $p_i$ and $p_j$ and cutting to create two configurations $p_i, p_{i+1}, \dots, p_j$ and $p_j, p_{j+1}, \dots, p_i$. Then either $p_i, p_{i+1}, \dots, p_j$ and $q_i, q_{i+1}, \dots, q_j$ have different coordinates or $p_j, p_{j+1}, \dots, p_i$ and $q_j, q_{j+1}, \dots, q_i$ have different coordinates. This is because if both pairs had the same coordinates, then gluing together would give that the larger configurations had the same coordinates. Then with the smaller configurations, we may again add in points on the perimeter if necessary and make more cuts. Eventually, we will be left with non-equivalent miniscule triangles, so that we would find that the configurations were not equivalent.

\end{proof}

Finally, the following lemma shows that in our definition of equivalence of virtual configurations--two virtual configurations $(p_i, \lambda_i)$ and $(q_i, \mu_i)$ are equivalent if there exists $\nu_i$ such that the actual positive configurations realizing $(p_i, \lambda_i+\nu_i)$ and $(q_i, \mu_i+\nu_i)$ are equivalent--any choice of $\nu_i$ such that $\lambda_i+\nu_i > 0$ and $\mu_i+\nu_i > 0$ will suffice. Thus have the alternative definition:

\begin{definition}
Two virtual configurations $(p_i, \lambda_i)$ and $(q_i, \mu_i)$ are equivalent if and only if for all $\nu_i$ such that $\lambda_i+\nu_i > 0$ and $\mu_i+\nu_i > 0$, the actual positive configurations realizing $(p_i, \lambda_i+\nu_i)$ and $(q_i, \mu_i+\nu_i)$ are equivalent.
\end{definition}

The lemma is now quite simple:

\begin{lemma} If $p_i$ and $q_i$ are equivalent positive configurations of $n$ points in the affine building, then for any $n$ dominant coweights $\lambda_1, \lambda_2, \dots, \lambda_n$, if we allow the $\lambda_i$ to act on $p_i$ and $q_i$, the resulting positive configurations will still be equivalent.
\end{lemma}

\begin{proof}
The action of the $\lambda_i$ on the configurations $p_i$ and $q_i$ changes the co-ordinates in an explicit way (see the end of section 5.4). We have just shown that configurations are equivalent if and only if they have the same coordinates. The configurations $p_i$ and $q_i$ have the same coordinates, so acting on them by the coweights $\lambda_i$ will give us configurations that still have the same coordinates. The resulting positive configurations will then be equivalent.

\end{proof}

Let us point out an interesting byproduct of Theorem~\ref{tropical points}. Note that the notion of equivalence between configurations of points in the affine building is weaker than (though possibly equivalent to) the notion of isometry. Thus different laminations can be distinguished using purely metric properties. Thus,

\begin{cor}\label{tropical distance functions} For a positive configuration of points $p_1, \dots p_n$ in the affine building, the tropical functions $f_{ijk}^t (p_a,p_b,p_c)$ depend only on the metric properties of the affine building. 
\end{cor}

We will investigate in a future paper how to construct the functions $f_{ijk}^t (p_a,p_b,p_c)$ explicitly in terms of the geometry of the affine building. This is the beginnings of the study of intersection pairings between laminations. One can further conjecture:

\begin{conj} For any configuration of points $p_1, \dots p_n$ in the affine building, the tropical functions $f_{ijk}^t (p_a,p_b,p_c)$ depend only on the metric properties of the configuration in the affine building. 
\end{conj}

\subsection{Relationship with the work of Goncharov-Shen}

After writing this paper, we learned of the very interesting and beautiful related work of Goncharov and Shen \cite{GS}. In this section, we would like to explain the relationship between our work and theirs.

In \cite{GS}, the authors define a distinguished cone of tropical points $\A^+_{G,S}(\Zt) \subset \A_{G,S}(\Zt)$ inside the set of all tropical points. This cone is cut out by a \emph{potential} $\mathcal{W}$ on the space $\A_{G,S}(\Zt)$. They give a bijection between the points in the cone and components of an object that they call the \emph{surface affine Grassmanian}. In particular, in the case where $S$ is a disc with marked points, they associate to tropical points of $\A^+_{G,S}(\Zt)$ top components of fibres of the convolution morphism, an object that arises naturally from Geometric Satake.

One of our motivations in our work was to attach an object to \emph{all} tropical points in $\A_{G,S}(\Zt)$. This was our reason for introducing the notion of ``virtual'' configurations. Another reason to work with ``virtual'' configurations is that it allows us to also tropicalize the space $\X_{G,S}$. To clarify the situation, we will show below that while $\A_{G,S}(\Zt)$ parameterizes virtual positive configurations of points in the affine building, $\A^+_{G,S}(\Zt)$ parameterizes positive configurations of points in the affine building. Thus the cone defined in \cite{GS} picks out the virtual configurations that are actual configurations. The duality conjectures of Fock and Goncharov imply that actual configurations on a disc with $n$ points parameterize invariants of $n$-fold tensor products of representations of $G$. This result follows from \cite{GS} (building on the work of Kamnitzer \cite{K}). We will give an alternative proof of the theorem below.

On the one hand, \cite{GS} relates tropical points to the geometry of the affine Grassmanian. Their approach has the advantage that it works for all semi-simple groups $G$. Their constructions are elegant, and their method is robust enough to cover many variations on the theme. Their work is powerful, but it depends crucially on the work of Kamnitzer \cite{K}.

On the other hand, our work focuses on the affine building rather than the affine Grassmanian. This has the advantage of clarifying the source of the piecewise-linear nature of the tropical functions: this tropical geometry comes from the geometry of the affine building (compare with \cite{JSY}). Our definitions are somewhat less abstract, while also having the flavor of geometric topology, and are more closely related to the classical theory of laminations. They also are related to spectral networks \cite{KNPS}. Finally, our notion of laminations is also concrete and simple: see the definitions in section 5.3. Our proofs, though the arguments are complicated in places, are in the end elementary, hands-on and self-contained.

Positive configurations of points in the affine building up to equivalence come from positive configurations in the affine Grassmanian. These positive configurations in the affine Grassmanian lie inside the top components of the convolution variety considered in \cite{GS}. One can think of our work as giving a positive structure on these convolution varieties (Note~\ref{positive structure}). In \cite{GS}, instead of using positive points, they use transcendental points, which have many formal similarities. One reason using positive points is useful is that we can then view laminations as coming from formal paths in the original space that limit to the tropical boundary. A more important reason is that the positive structure reflects some kind of planar structure on higher laminations, for example see Lemma~\ref{adding perimeter points}. In particular, for $G=SL_2$, we see that positivity is reflected in the fact that the trees we obtain are \emph{planar} trees.

Let us now relate the two approaches. In the course of the proof, we will give an alternative derivation of the hive inequalities independent of the approaches in \cite{K} and \cite{GS}.

\begin{theorem} Let $G=SL_m$. Then the points of $\A^+_{G,S}(\Zt)$ correspond to positive configurations of points in the affine building.
\end{theorem}

For $G=SL_m$, the points of $\A^+_{G,S}(\Zt)$ are precisely those (\cite{GS}) such that on each triangle we have that 
$$f_{i+1,j,k}^t+f_{i-1,j+1,k+1} \leq f_{i,j+1,k}^t+f_{i,j,k+1},$$
and also the similar the inequalities obtained by permutation of indices,
$$f_{i,j+1,k}^t+f_{i,j-1,k} \leq f_{i+1,j,k}^t+f_{i,j,k+1},$$
$$f_{i,j,k+1}^t+f_{i,j,k-1} \leq f_{i,j+1,k}^t+f_{i+1,j,k}.$$
These are the famous \emph{hive} or \emph{rhombus} inequalities. Here we use the convention that 
$$f_{m00}^t=f_{0m0}^t=f_{00m}^t=0.$$

We will treat the case of a triangle. Because of Lemma~\ref{independent of triangulation}, we see that if on each triangle we have a actual positive configuration of points in the affine building, then the entire configuration is an actual positive configuration of points in the affine building. Thus the case of the triangle implies the general case.

Suppose we have some virtual configuration $(x_1,\lambda_1), (x_2,\lambda_2), (x_3,\lambda_3)$ with coordinates $f_{ijk}^t$. We would like to see under what circumstances $(x_1,\lambda_1), (x_2,\lambda_2), (x_3,\lambda_3)$ is equivalent to an actual configuration $x'_1, x'_2, x'_3$. We can in fact analyze each of $(x_1,\lambda_1), (x_2,\lambda_2), (x_3,\lambda_3)$ individually to see whether each is a virtual point or an actual point. In other words, we can say something more refined:

\begin{prop} Let $(x_1,\lambda_1), (x_2,\lambda_2), (x_3,\lambda_3)$ be a virtual positive configuration of points in the affine building. Then it is equivalent to the configuration $(x_1,\lambda_1), (x_2,\lambda_2), (x_3',0)$ if and only if
\begin{multline}
f_{i,j,k+1}^t((x_1,\lambda_1), (x_2,\lambda_2), (x_3,\lambda_3))+f_{i,j,k-1}((x_1,\lambda_1), (x_2,\lambda_2), (x_3,\lambda_3)) \leq \\ f_{i,j+1,k}^t((x_1,\lambda_1), (x_2,\lambda_2), (x_3,\lambda_3))+f_{i+1,j,k}((x_1,\lambda_1), (x_2,\lambda_2), (x_3,\lambda_3)).
\end{multline}
In other words, we may take $x_3$ to be an actual point if and only if the above inequalities hold. Similar statements hold for $(x_1,\lambda_1), (x_2,\lambda_2)$ and the other sets of inequalities. 
\end{prop}

\begin{proof}

It is easy to see that replacing $(x_1,\lambda_1), (x_2,\lambda_2)$ with $x_1, x_2$ does not affect the inequalities in the proposition (in fact, it changes both sides of each inequality by the same amount, as can be seen by direct calculation). Thus it is sufficient to prove this in the case that we have a virtual configuration $x_1, x_2, (x_3,\lambda_3)$.

Suppose that $$f^t_{ijk}(x_1,, x_2, (x_3,\lambda_3))=a_{ijk}.$$
Now let $\lambda$ be any dominant coweight, and let $$b_i=\omega_i \cdot \lambda.$$

Let us construct an auxiliary point $x_4$ in the building with the property that 
$$f^t_{ijk}(x_1, (x_3, \lambda_3), x_4)=a_{i, 0, m-i}+b_k.$$
We choose $x_4$ so that $x_1, x_2, (x_3,\lambda_3), x_4$ is a virtual positive configuration of points in the affine building. We can see that $x_4$ is an actual point in the building by using the fact that $\lambda$ is dominant and Lemma~\ref{independent of triangulation}. The values of the coordinates $f^t_{ijk}(x_1, (x_3, \lambda_3), x_4)$ are chosen so that when $(x_3, \lambda_3)$ can be represented by an actual point $x_3'$, we have that $x_1, x_3'$ and $x_4$ lie on a geodesic.

The proof then comes down to two lemmas:

\begin{lemma} The virtual configuration $x_1, x_2, (x_3, \lambda_3), x_4$ is equivalent to the actual configuration $x_1, x_2, x_3', x_4$ if and only if

$$f^t_{ijk}(x_1, x_2, x_4)=a_{ijk}+b_k,$$
$$f^t_{ijk}(x_2, (x_3, \lambda_3), x_4)=a_{0, i, m-i}+b_k.$$

\end{lemma}

\begin{lemma} Given that
$$f^t_{ijk}(x_1,, x_2, (x_3,\lambda_3))=a_{ijk}$$
and
$$f^t_{ijk}(x_1, (x_3, \lambda_3), x_4)=a_{i, m-i, 0}+b_k,$$
then we have
$$f^t_{ijk}(x_1, ,x_2, x_4)=a_{ijk}+b_k,$$
$$f^t_{ijk}(x_2, (x_3, \lambda_3), x_4)=a_{0, i, m-i}+b_k$$
if and only if 
\begin{multline}
f_{i,j,k+1}^t(x_1, x_2 (x_3,\lambda_3))+f_{i,j,k-1}(x_1, x_2 (x_3,\lambda_3)) \leq \\ f_{i,j+1,k}^t(x_1, x_2 (x_3,\lambda_3))+f_{i+1,j,k}(x_1, x_2 (x_3,\lambda_3)).
\end{multline}
\end{lemma}

The second lemma is a straightforward computation using the tropicalized mutation relations, which relate the co-ordinates in one triangulation to the co-ordinates in another.

Let us prove the first lemma. First suppose that $(x_3,\lambda_3)$ is represented by the actual point $x_3'$. Then by the construction of $x_4$, we have that $x_1, x_3', x_4$ lie on a geodesic in that order, with $d(x_4, x_3')=\lambda$. Then it is clear that $f^t_{ijk}(x_1, x_2, x_4)$ and $f^t_{ijk}(x_2, (x_3, \lambda_3), x_4)$ have co-ordinates as above.

On the other hand, if $f^t_{ijk}(x_1, x_2, x_4)$ and $f^t_{ijk}(x_2, (x_3, \lambda_3), x_4)$ have co-ordinates as above, then if we let $x_3'$ be the point along the geodesic between $x_1$ and $x_4$ such that $d(x_4, x_3')=\lambda$, then 
$$f^t_{ijk}(x_2, x_3', x_4)=a_{0, i, m-i}+b_k=f^t_{ijk}(x_2, (x_3, \lambda_3), x_4).$$
Gluing to the triangle formed by $x_1, x_2, x_4$, we find that the configuration $x_1, x_2, (x_3, \lambda_3), x_4$ is equivalent to the actual configuration $x_1, x_2, x_3', x_4$.

\end{proof}

\subsection{Laminations for the $\A$-space of a general surface}

Our goal now is to define laminations on a general surface $S$, possibly with marked points. The basic idea is that any surface can be glued together from triangles. Once we know laminations on a triangle and how to glue together laminations along an edge, we know how to construct laminations on a general surface. There is one subtlety, which we will discuss below.

We have defined laminations on $S$ when $S$ is a disc with marked points. We will sometimes refer to such $S$ as a ``polygon.'' Proposition~\ref{tropical relations} shows that coordinates on laminations on a polygon satisfy the tropical relations, while Theorem~\ref{tropical points} shows that these coordinates determine the lamination completely. Thus for $S$ a disc with marked points, we have an identification between tropical points of $\A_{G,S}$ and $G$-laminations on $S$.

It is clear that our construction is compatible with cutting and gluing of polygons; laminations are completely determined by their coordinates, and the coordinates are constructed locally with respect to a triangulation, with coordinates associated to the edges and triangles of a triangulation. We now wish to extend this to surfaces. We will first analyze the case of polygons more closely.

Suppose we have two laminations $$(p_1,\lambda_1), (p_2,\lambda_2), \dots, (p_n,\lambda_n)$$ and $$(q_1,\mu_1), (q_2,\mu_2), \dots, (q_l,\mu_l)$$ on an $n$-gon and an $l$-gon. We may glue laminations on these polygons (in a unique way) by gluing the edge $(p_1, \lambda_1)$, $(p_2,\lambda_2)$ to the edge $(q_2, \mu_2)$, $(q_1,\mu_1)$ if and only if the two virtual configurations of two points  $(p_1, \lambda_1)$, $(p_2,\lambda_2)$ and $(q_2, \mu_2)$, $(q_1,\mu_1)$ are equivalent. If they are, by choosing  $\lambda'_1$ to be a coweight larger than $\lambda_1$ and $\mu_2$ and by choosing $\lambda'_2$ to be a coweight larger than $\lambda_1$ and  $\mu_1$, we get a virtual configuration of two points $(p'_1,\lambda'_1)$, $(p'_2,\lambda'_2)$ such that $$(p'_1, \lambda'_1), (p'_2,\lambda'_2), (p_3, \lambda_3), \dots, (p_n,\lambda_n)$$ is equivalent to $$(p_1,\lambda_1), (p_2,\lambda_2), \dots, (p_n,\lambda_n).$$ Similarly, we have that $$(q_1,\mu_1), (q_2,\mu_2), \dots, (q_l,\mu_l)$$ is equivalent to $$(q_1,\lambda'_2), (q_2,\lambda'_1), \dots, (q_l,\mu_l).$$
Additionally, we may translate this configuration so that $(q_1,\lambda'_2), (q_2,\lambda'_1)$ and $(p'_1,\lambda'_1)$, $(p'_2,\lambda'_2)$ coincide by moving $q_i$, $i>2$ to obtain a virtual configuration $$(p'_2, \lambda'_2), (p'_1, \lambda'_1), (q'_3, \mu_3), \dots, (q'_l,\mu_l)$$ which is equivalent to $$(q_1, \mu_1), (q_2,\mu_2), \dots, (q_l,\mu_l).$$ Then the we can glue to get a lamination $$(p'_2,\lambda'_2), (p_3, \lambda_3), \dots, (p_n,\lambda_n), (p'_1,\lambda'_1), (q'_3, \mu_3), \dots, (q'_l,\mu_l)$$ on an $n+l-2$-gon.

The result of this discussion is that if we are gluing polygons, we may have to replace one of the virtual points of our configuration, $(p, \lambda)$, with another virtual point, $(p', \lambda')$, that gives an equivalent configuration. The reason is that if a vertex belongs to several different triangles, the value of $\lambda$ at this vertex may be different in each triangle, and we have to choose a $\lambda$ large enough for all the triangles containing this vertex simultaneously. 

Lemma~\ref{independent of triangulation} gave us a way to explicitly find such a $\lambda$: Start with a configuration of $n$ flags in $G(\cK)/U(\cK)$, and choose some triangulation of the $n$-gon. Then if a vertex belongs to $r$ triangles, the virtual point at the vertex in each of these triangles might be $(p,\lambda_i)$, $i=1, 2, \dots, r$. By choosing $\lambda$ larger than all the $\lambda_i$, we may replace all these $(p,\lambda_i)$ by a single virtual point $(p,\lambda)$.

By Lemma~\ref{independent of triangulation}, this choice of $\lambda$ then works to give us a virtual configuration of $n$ points regardless of the triangulation. Thus we know that we just have to choose $\lambda$ larger than the value which is necessary for each of the triangles to which a vertex belongs. It turns out that in a general surface the same holds, though this is not obvious a priori, as we will be defining laminations on a surface via {\em infinite} virtual positive configurations of points in the affine building, and therefore each vertex belongs to infinitely many triangles. We discuss this below.

For a general surface $S$ with marked points, we consider the cyclic set at  $\infty$ formed by all the lifts of these marked points to the universal cover of $S$. This set $C$ is infinite, and comes with a free action of $\pi_1(S)$. Then we may define

\begin{definition} A $G$-lamination on a surface $S$ with marked points is a virtual positive cyclic configuration of points in the affine building parameterized by the set $C$, equipped with an equivariant action of $\pi_1(S)$.
\end{definition}

For more on the definition of this cyclic set, see section 2 of this paper or the introduction of \cite{FG1}. An equivariant action of $\pi_1(S)$ means that for any $\gamma \in \pi_1(S)$, pulling back the configuration by the map $\gamma$ gives an equivalent configuration.

Almost everything that holds for a finite virtual positive configuration carries over to the infinite case. We shall say that for an infinite set of points, two positive configurations parameterized by this cyclic infinite set are equivalent if for every finite subset, the configurations are equivalent.

One might worry that because we have an infinite number of points, there is not suitable choice of a large enough $\lambda$ at a given vertex. But if we start with a triangulation of a surface and lift this to the universal cover, although there are an infinite number of triangles at each vertex, there are only a finite number triangles up to the action of of the stabilizer of this vertex inside $\pi_1(S)$. Thus, there are only a finite number of values $\lambda_i$ that constrict our choice of $\lambda$ at this vertex. Thus we have

\begin{theorem} Let $S$ be a (hyperbolic) surface with marked points, and let $C$ be its cyclic set at $\infty$. Associated to any tropical point of $\A_{G,S}(\Zt)$ there is a $\pi_1(S)$-equivariant virtual positive configuration of points in the affine building of $G$ parameterized by $C$. This configuration is unique up to equivalence.
\end{theorem}

The tropical coordinates on this lamination come from a triangulation of the surface $S$. We lift this triangulation to a triangulation of the disc with the cyclic set $C$ at the boundary. On each triangle, we have a virtual positive configuration $(p_i, \lambda_i)$ of points in the affine building, which comes from a virtual positive configuration of points $(x_i, \lambda_i)$ in the affine Grassmanian, which in turn comes from a positive configuration of flags $F_i \in G/U^-(\cK)$. The tropical functions are

$$f_{ijk}^t ((p_a,\lambda_a),(p_b\lambda_b),(p_c\lambda_c)) :=f_{ijk}^t ((x_a,\lambda_a),(x_b\lambda_b),(x_c\lambda_c)) = \val(f_{ijk} (F_a,F_b,F_c)).$$

These functions satisfy the tropical relations, and therefore completely determine lamination.

\section{Laminations for the $\X$ space}

\subsection{Laminations for the $\X$-space of a disc}

\subsubsection{Configurations of cones}
We will now treat the dual case of laminations for the $\X$ space. For the $\X$ space, we now work with the group $G=PGL_n$. Let $S$ be a polygon. Recall that any positive configuration $p_1, p_2, \dots, p_n$ of points in the affine building comes in a family of equivalent virtual configurations  $(q_1,\lambda_1), (q_2,\lambda_2), \dots, (q_n,\lambda_n)$, where the $q_i$ vary with the $\lambda_i$, and $(p_1,0), (p_2,0), \dots, (p_n,0)$ is among these equivalent virtual configurations. We will now define a related concept which has a slightly different flavor.

Start with a positive configuration of $n$ (principal) flags $F_1, \dots, F_n$, $F_i \in GL_n/U(\cK)$. Note that here we no longer require that the flags have determinant $1$. As before, we can choose  $v_{i1}, \dots, v_{im}$ $1 \leq i \leq n$ to be some lifts of these flags to $GL_n(\cK)$, then choose $\lambda_i$, such that $$t^{-\lambda_{i1}}v_{i1}, \dots, t^{-\lambda_{im}}v_{im}$$ is a good lift of $F_i \cdot \lambda_i$. Here the $\lambda_i$ are coweights for $GL_n$.

Let $x_i$ be the $\cO$-module spanned by $t^{-\lambda_{i1}}v_{i1}, \dots, t^{\lambda_{im}}v_{im}$. We view $x_i$ inside the affine Grassmanian for $PGL_n$ by considering this $\cO$-module up to scaling by any element of $\cK$. We will only care about the $\lambda_i$ up to their image in the coweight lattice for $PGL_n$ (as opposed to the coweight lattice for $GL_n$) and we retain the notation $\lambda_i$ for the image in the coweight lattice of $PGL_n$. Then we can associate to $F_1, \dots, F_n$ the virtual positive configuration $(x_1,-\lambda_1), \dots, (x_n,-\lambda_n)$, where the $x_i$ form a positive configuration inside the affine Grassmanian for $PGL_n$. We can then map to the affine building to get a virtual positive configuration $$(q_1,-\lambda_1), (q_2,-\lambda_2), \dots, (q_n,-\lambda_n).$$ Different choices of $\lambda_i$ give different configurations $q_i$ which trace out a cone in the affine building. We would like to forget the data of parameterization by $\lambda_i$ and just think about the asymptotics of these cones. In other words, we are most interested in the positive configurations $q_1, \dots, q_n$ that arise as the $\lambda_i$ get large. This motivates the following definition.

Let us choose a base virtual configuration $$(q_1^0,-\lambda_1^0), (q_2^0,-\lambda_2^0), \dots, (q_n^0,-\lambda_n^0)$$ coming from $F_1, \dots, F_n$. Then as the $\lambda_i$ varies over the range $\lambda_i \geq \lambda_i^0$, $q_i$ varies over a subset $Q_i$ of the building which we call a {\emph cone}. We will say that these cones are {\emph based) at $q_i^0$.

\begin{definition} A {\it positive configuration of cones} in the affine building consists of $n$ cones of points $Q_1, Q_2, \dots, Q_n$, where each $Q_i$ is based at $q_i^0$. If we have a set of representative elements $q_i \in Q_i$, we will say that the positive configuration of points $q_1, \dots, q_n$ is contained in the positive configuration of cones $Q_1, \dots, Q_n$, or that $q_1, \dots, q_n$ is a representative positive configuration of points inside the positive configuration of cones $Q_1, \dots, Q_n$.
\end{definition}

\begin{rmk} $Q_i$ is exactly the set of points that comes from the action of the moniod $\Lambda_+$ on $q_i^0$.
\end{rmk}

Whereas before we were interested in virtual positive configurations, we now are interested in these positive configurations of cones. 

We will consider these positive configurations of cones up to equivalence. We define two configurations of cones $P_i$ and $Q_i$ to be equivalent if they contain {\it some} representative configurations $p_1, p_2, \dots, p_n$ and $q_1, q_2, \dots, q_n$ which are equivalent in the previous sense (which we defined when discussing laminations for the $\A$ space). Let us elaborate on this. Consider the configurations of sub-cones of of the configurations of cones $P_i$ and $Q_i$ based at $p_i$ and $q_i$ respectively. Then the corresponding representative configurations of points within these sub-cones are all equivalent, by our previous analysis of actions of $\Lambda^+$ on positive configurations of points in the affine building. Thus we are really studying the {\em asymptotics} of these cones.

Note that we lose some degrees of freedom in passing from virtual configurations to configurations of cones because we forget the data of the $\lambda_i$. It is as if we were considering equivalence classes of virtual positive configurations of $n$ points modulo the action of $\Lambda^n$.

\subsubsection{Functions on $\X_{G,S}$} We now recall the functions on the space $\X_{G,S}$ for $S$ disc with marked points, and also describe their tropicalization. $\X_{G,S}$ is the space of configurations of points in $G/B$. The functions on $\X_{G,S}$ come from a triangulation, and there are functions attached to each face and each interior (non-boundary) edge of the triangulation. For each face of a triangulation of the $n$-gon, we have a set of functions $g_{ijk}$, where $i+j+k=m$ and $i, j, k > 0$. Suppose that some triangle has flags $B_1, B_2, B_3$ at its vertices, where $B_i \in G/B$. Then

$$g_{ijk} (B_1,B_2,B_3) = \frac{f_{i-1,j+1,k}(F_1,F_2,F_3) \cdot f_{i,j-1,k+1}(F_1,F_2,F_3) \cdot f_{i+1,j,k-1}(F_1,F_2,F_3)}{f_{i+1,j-1,k}(F_1,F_2,F_3) \cdot f_{i,j+1,k-1}(F_1,F_2,F_3) \cdot f_{i-1,j,k+1}(F_1,F_2,F_3)}$$
where $F_1, F_2, F_3$ are any lifts of $B_1, B_2, B_3$ from $G/B$ to $GL_n/U^{-}$.

Moreover, for any edge of the triangulation, we can look at the two triangles it belongs to, and construct a set of functions $g_{ij}$, for $i+j=m$ and $i, j> 0$. If the four flags $B_1, B_2, B_3, B_4$ form two triangles $B_1, B_2, B_3$ and $B_3, B_4, B_1$ which share the edge $B_1, B_3$, we have the functions 

$$g_{ij} (B_1,B_2,B_3,B_4) = \frac{f_{i-1,1,j}(F_1,F_2,F_3) \cdot f_{j-1,1,i}(F_3,F_4,F_1)}{f_{i,1,j-1}(F_1,F_2,F_3) \cdot f_{j,1,i-1}(F_3,F_4,F_1)}$$
where $F_1, F_2, F_3, F_4$ are any lifts of $B_1, B_2, B_3, B_4$ from $G/B$ to $GL_n/U^{-}$.

It is shown in \cite{FG1} that these functions are independent of the choice of these lifts to $GL_n/U^{-}$ for both the face and edge functions. In particular, replacing $F_i$ with $F_i \cdot \lambda_i$ does not change the value of these functions.

We will now define the tropicalization of these functions. For a positive configuration of cones $P_1, P_2, \dots, P_n$ in the affine building, choose some representative configuration $p_1, p_2, \dots, p_n$. Suppose that this came from taking a positive configuration of flags $F_i$, choosing large enough $\lambda_i$, forming the corresponding virtual configuration $(x_i,\lambda_i)$ of points in the affine Grassmanian, and then taking the corresponding virtual positive configuration $(p_i,\lambda_i)$ of points in the affine building. Then if a triangle has $p_a, p_b, p_c$ at its vertices, where the $p_i$ are points of the building, then we define
$$g_{ijk}^t (p_a,p_b,p_c) = (f_{i-1,j+1,k}^t + f_{i,j-1,k+1}^t + f_{i+1,j,k-1}^t - f_{i+1,j-1,k}^t - f_{i,j+1,k-1}^t - f_{i-1,j,k+1}^t) (p_a,p_b,p_c)$$
for $i+j+k=m$ and $i, j, k > 0$. One can make a similar definition for $g_{ijk}^t (x_a,x_b,x_c)$.

Let us try analyze the lifting involved in defining this function. Recall that
$$f_{ijk}^t ((p_a,\lambda_a),(p_b,\lambda_b),(p_c,\lambda_c)) = -\val(f_{ijk} (F_a,F_b,F_c))$$
so that
$$f_{ijk}^t (p_a,p_b,p_c)=-\val(f_{ijk} (F_a \cdot -\lambda_a,F_b \cdot -\lambda_,F_c \cdot -\lambda_c)).$$
Then we have by an easy calculation that
$$g_{ijk}^t (p_a,p_b,p_c) = -\val(g_{ijk} (F_a \cdot -\lambda_a,F_b \cdot -\lambda_,F_c \cdot -\lambda_c)) = -\val(g_{ijk} (F_a,F_b,F_c)),$$
so it is clear that $g_{ijk}^t$ does not depend on the choice of representative of the configuration of cones.

Similarly, for a positive configuraiton of four points in the building $p_a, p_b, p_c, p_d$, we define
$$g_{ij}^t (p_a,p_b,p_c,p_d) := f_{i-1,1,j}^t (p_a,p_b,p_c) + f_{j-1,1,i}^t (p_c,p_d,p_a) - f_{i,1,j-1}^t (p_a,p_b,p_c) - f_{j,1,i-1}^t (p_c,p_d,p_a)$$
for $i+j=m$ and $i, j> 0$. One can similarly define $g_{ij}^t (x_a,x_b,x_c,x_d)$ 

One again shows that
$$g_{ij}^t (p_a,p_b,p_c,p_d) = -\val(g_{ij}(F_a,F_b,F_c,F_d))$$
and that therefore the choice of configuration $p_i$ within the configuration of cones does not affect the values of these functions.

The following theorem is similar to the $\A$ case:

\begin{theorem} Let $S$ be a disk with $n$ marked points. There is a bijection between tropical points of $\X_{G,S}(\Zt)$ and positive configurations of cones in the affine building up to equivalence.

Given a positive configuration of cones, take one representative configuration of points $p_1, p_2, \dots, p_n$ in the affine building. It comes from positive configuration of points $x_i$ in the affine Grassmanian, which in turn comes from a positive configuration of flags $F_i \in GL_n/U^-(\cK)$ and a choice of large enough coweights $\lambda_i$. Given a trianglulation of the $n$-gon, we associate to each triangle $p_a, p_b, p_c$ the functions $g_{ijk}^t (p_a,p_b,p_c)$, and to each edge $p_a, p_c$ bordering two triangles $p_a, p_b, p_c$ and $p_b, p_c, p_d$, we associate the functions $g_{ij}^t (p_a,p_b,p_c,p_d)$.

These functions satisfy the tropical relations, and therefore give a well-defined point of  $\X_{G,S}(\Zt)$. The values of these functions completely determine the family of positive configurations of points in the affine building up to equivalence, and vice versa.

\end{theorem}

\begin{proof}

Starting with a point of $\X_{G,S}(\Zt)$ we will construct a positive configuration of $n$ cones in the affine building for $G$. Consider the space $\X_{G,S}(\cK_{>0})$. We can choose some triangulation, and then choose as we like the values of $g_{ijk}$ and $g_{ij}$ for each triangle and edge. In particular, we choose the values of these functions to lie in $\cK_{>0}$, such that they have prescribed valuations corresponding to our point of $\X_{G,S}(\Zt)$. This is always possible, and gives us a point of $\X_{G,S}(\cK_{>0})$. This gives a positive configuration of points $B_1, \dots, B_n$ in $G/B(\cK)$. The results of \cite{FG1} show that one can lift this to a positive (twisted) configuration of points in $G/U^{-}(\cK)$ (this is not explicitly stated, but comes from examining equation 5.2 on page 73, theorem 7.3 on page 96, and equation 8.9 on page 119 of \cite{FG1}). On this configuration of flags $F_1, \dots, F_n$, the functions $f_{ijk}$ are defined, and they are related to the functions $g_{ijk}$ and $g_{ij}$ by the relations above above.

The configuration of flags gives rise to a configuration of cones in the affine building out of which we can choose one representative configuration $p_1, \dots, p_n$. The values of $g_{ijk}^t$ and $g_{ij}^t$ on $p_1, \dots, p_n$ will be negative the valuations of the values of $g_{ijk}$ and $g_{ij}$ on $F_1, \dots, F_n$ by construction. Moreover, these values will be independent of our choice of representatives of these cones, because we showed that the values of $g_{ijk}^t$ and $g_{ij}^t$ were independent of $\lambda_i$.

Thus $g_{ijk}^t$ and $g_{ij}^t$ really are functions of the positive configuration of cones $P_i$ containing $p_1, \dots, p_n$, and not just of the configuration itself.

Thus starting from any point of $\X_{G,S}(\cK_{>0})$ we obtain a positive configuration of cones in the affine building. The functions $g_{ijk}^t$ and $g_{ij}^t$ on these positive configurations of cones give a positive configuration of cones corresponding to each tropical point of $\X_{G,S}(\Z^t)$. As explained in section 2, the co-ordinate transformations under changes of triangulation behave tropically, and we get a well-defined point of $\X_{G,S}(\Z^t)$.

It is clear that any two positive configurations of cones which are equivalent have the same values for the functions $g_{ijk}^t$ and $g_{ij}^t$, because the cones contain some equivalent configurations on which the functions $f_{ijk}^t$ take identical values, and the functions $g_{ijk}^t$ and $g_{ij}^t$ are determined by the functions $f_{ijk}^t$.

Finally, we need to show that different points of $\X_{G,S}(\cK_{>0})$ which have the same valuations give equivalent configurations of cones. Suppose we have two positive configurations of cones on which the functions $g_{ijk}^t$ and $g_{ij}^t$ are identical. Then let $p_1, \dots, p_n$ and $q_1, \dots, q_n$ be representatives of these cones. Then although the functions $f_{ijk}^t$ are not determined by $g_{ijk}^t$ and $g_{ij}^t$, they are determined up to the action of $\Lambda^n$. Therefore there are two other configurations $p_1', \dots, p_n'$ and $q_1', \dots, q_n'$ that come from applying the action of $\Lambda_+^n$ on $p_1, \dots, p_n$ and $q_1, \dots, q_n$, respectively, such that the functions $f_{ijk}^t$ have identical values on the configurations $p_1', \dots, p_n'$ and $q_1', \dots, q_n'$. Because positive configurations are determined by the functions $f_{ijk}^t$, $p_1', \dots, p_n'$ and $q_1', \dots, q_n'$ are equivalent as positive configurations of points in the affine building, and therefore our original configurations of cones in the affine building are equivalent.

\end{proof}

We now discuss and attempt to clarify the idea behind $\X$-laminations. In our presentation, the definition of the functions $g_{ijk}$ and $g_{ij}$, tropical or not, depended on some choice: on higher Teichmuller space, the functions depended on the choice of a principal flag (an element of $G/U$) dominating a regular flag (an element of $G/B$), while on its tropicalization, the functions depend on a configuration of points representing a configuration of cones. The choice of a dominating principal flag or a point within a cone is analogous to the choice of a horocycle at a cusp in classical Teichmuller theory. The functions are defined using principal flags or a configuration of points, but they ultimately only depend on the flags and configuration of cones.

This means that we lose some degrees of freedom in going from $\A_{G,S}$ to $\X_{G,S}$. For example, on a triangle, the space $\X_{G,S}$ doesn't have co-ordinates corresponding to the boundary edges. On the other hand, the gluings between triangles are more interesting: for the space $\A_{G,S}$ we can only glue two triangles if they have the same edge functions, where as for the space $\X_{G,S}$ any two triangles have an interesting space of gluings based on the co-ordinates assigned to the edges. The different possible gluings are related to each other by shearing. This is one important feature that distinguishes the spaces $\A_{G,S}$ and $\X_{G,S}$. This will become even more important in the next section.

\subsection{Laminations for the $\X$-space of a general surface}

We now define $\X$-laminations on a general open surface $S$, possibly with marked points. Again, because any open surface can be glued together from triangles, knowing laminations on a triangle and how they can glue together to get laminations on a quadrilateral will allow us to construct laminations on a general surface.

We have defined laminations on $S$, where $S$ is a ``polygon." The previous section gave an identification between tropical points of $\X_{G,S}$ and $G$-laminations on $S$.

It is clear that our construction is compatible with cutting and gluing of polygons. In the $\A$ case, we took some care to show that in gluing, we didn't have to increase the values of $\lambda_i$ indefinitely. However, because here we are interested in positive configurations of cones, we do not have this concern.

For a general surface $S$ with marked points, recall the cyclic set at  $\infty$ formed by all the lifts of these marked points to the universal cover of $S$. This set $C$ is infinite, and comes with a free action of $\pi_1(S)$. Then we may define

\begin{definition} A $G$-lamination for the $\X$-space of a surface $S$ with marked points is the data of a positive cyclic configuration of cones in the affine building parameterized by every finite subset of the set $C$, compatible under restriction from one finite set to another, and equipped with an action of $\pi_1(S)$ on these configurations of cones.
\end{definition}

We should understand the action of $\pi_1(S)$ as follows. If some element $\gamma$ of the fundamental group moves the set finite subset $S \subset C$ to the finite subset $S' \subset C$, then the configuration of cones on $S$ is equivalent to the configuration of cones on $S'$.

Almost everything that holds for finite positive configurations of cones carries over to the infinite case. We shall say that two $G$ laminations for the $\X$ space are equivalent if for any finite subset of $C$, the two positive configurations of cones parameterized by this finite set are equivalent.

\begin{theorem} Let $S$ be a (hyperbolic) surface with marked points, and let $C$ be its cyclic set at $\infty$. Associated to any tropical point of $\X_{G,S}(\Zt)$ there is a $G$-lamination as defined above. This $G$-lamination is unique up to equivalence.

The tropical coordinates on this lamination come from a triangulation of the surface $S$. We lift this triangulation to a triangulation of the disc with the cyclic set $C$ at the boundary. On each triangle or quadrilateral, we have a positive configuration of cones within which we can take a representative positive configuration $p_i$ of points in the affine building, which comes from a positive configuration of points $x_i$ in the affine Grassmanian, which in turn comes from a positive configuration of flags $F_i \in GL_n/U^-(\cK)$. The tropical functions are

$$g_{ijk}^t (p_a,p_b,p_c) = g_{ijk}^t (x_a,x_b,x_c) = -\val(g_{ijk} (F_a,F_b,F_c))$$

and

$$g_{ij}^t (p_a,p_b,p_c,p_d) = g_{ij}^t (x_a,x_b,x_c,x_d) = -\val(g_{ij} (F_a,F_b,F_c,F_d)).$$

These functions satisfy the tropical relations, and therefore completely determine lamination.

\end{theorem}

We give a word of caution. We cannot talk about the family of positive configurations of points parameterized by the entire set $C$. This is because of the way gluing works. Suppose we had a positive configuration of four cones. Let $p_1, p_2, p_3, p_4$ a representative configuration. Then it restricts on the triangle $123$ to the configuration $p_1, p_2, p_3$. However, not every configuration of cones containing $p_1, p_2, p_3$ comes from a configuration of cones containing $p_1, p_2, p_3, p_4$. Thus in gluing together configurations of cones, we may need to replace a cone that has its origin at a point $p_i$ by one that has its origin at another point inside the original cone. In terms of gluing representative configurations of points, we may have to replace one configuration by a larger configuration representing the same cone (one obtained by acting on the original configuration of points by an element of $\Lambda_+^n$) in order to glue.

Thus the problem that we avoided for the $\A$-space turns out to be important here: as we take larger and larger subsets of $C$, the actual configurations that are representative of our positive configurations of cones may get larger and larger, so that in the limiting case we don't actually have a configuration of points. Thus we can only choose representative configurations for configurations of cones parameterized by any finite subset of $C$.

Moreover, the more robust gluing allowed in the case of laminations on the $\X$-space allows us to have interesting monodromy around a hole in our surface, unlike the situation in the case of the $\A$-space. This is a reflection of the fact that edge co-ordinates allow for a kind of ``shearing." Let $x$ be a point of $C$, the boundary at infinity. Then if $\gamma \in \pi_1(S)$ preserves $x$, it may not preserve the points of the affine building in the cone attached to the point $x$; it may move $p_x$ to another point $p'_x$, where $p_x$ and $p'_x$ are related by the action of $\Lambda$ at the point $x$ (so that the cones based at $p_x$ and $p'_x$ overlap and are asymptotic).

\subsection{Laminations on a closed surface}

In this section, we give a proposed definition for laminations on a closed surface. However, the definition is not as easy to work with as in the case of surfaces with boundary. We will try to make the difficulties clear below.

The general approach for the definition of laminations on a closed surface is similar to the definition for closed surfaces: we consider the space $\X_{G,S}(\cK_{>0})$ of positive $G(\cK)$-bundles, and then take valuations to obtain the tropical points $\X_{G,S}(\Zt)$, which parameterize higher laminations. Our first task will be defining the set $\X_{G,S}(\cK_{>0})$ of positive $G(\cK)$-bundles in the case where the surface $S$ is closed.

Recall that $\X_{G,S}(\cK_{>0})$ was defined for a surface with boundary by forming the cyclic set ${\mathcal F}_{\infty}$ consisting of lifts of all cusps and marked points to the universal cover of $S$. This is a set with a cyclic order sitting at the boundary of the universal cover of $S$. Points of $\X_{G,S}(\cK_{>0})$ were then identified with positive configurations of flags parameterized by ${\mathcal F}_{\infty}$ along with an equivariant action of $\pi_1(S)$.

For a closed surface $S$, the set ${\mathcal F}_{\infty}$ is empty. Instead of using the cyclic set ${\mathcal F}_{\infty}$, we could have instead used the larger cyclic set ${\mathcal G}_{\infty}={\mathcal F}_{\infty} \cup {\mathcal G}'_{\infty}$, where ${\mathcal G}'_{\infty}$ consists of preimages on the boundary at infinity of all endpoints of non-boundary geodesics on $S$. The set ${\mathcal G}_{\infty}$ carries a natural action of $\pi_1(S)$.

We can then define the space $\X_{G,S}(\cK_{>0})$.

\begin{definition} $\X_{G,S}(\cK_{>0})$ is the space of $\pi_1(S)$-equivariant positive configurations of points in $G/B(\cK)$ parameterized by ${\mathcal G}_{\infty}$.
\end{definition}

It would be natural to try to define laminations for a closed surface $S$ as follows: A $G$-lamination for a closed surface $S$ is the data of a positive cyclic configuration of cones in the affine building parameterized by every finite set of the set ${\mathcal G}_{\infty}$, compatible under restriction from one finite set to another, and equipped with an action of $\pi_1(S)$ on these configurations of cones, up to equivalence. These positive cyclic configurations of cones are considered up to equivalence. However, this definition does not reduce to our previous definition in the case where $S$ has boundary, and for $G=SL_2$, it does not recover Thurston's measured laminations. The problem is that the definition is too refined--there are laminations that should be equivalent under our definition, but are not.

Thus we will have to make some more considerations before proceeding. To understand better the case of closed surfaces, we use the cutting and gluing properties of higher Teichmuller spaces. We follow here the treatment given in sections 6.9, 7.6-7.9 of \cite{FG1} and give tropical analogues of those arguments.

Let $S$ be a surface with or without boundary, and suppose we have a point $x$ of $\X_{G,S}(\cK_{>0})$. Then along any any closed (oriented) curve $\gamma$ on $S$ we may calculate the monodromy $\rho(\gamma)$ of the local system around that closed curve.

By the results of \cite{FG1}, we know that $\rho(\gamma)$ will lie in $G(\cK_{>0})$. A result of Lusztig (\cite{Lu2}, Theorem 5.6) gives us that

\begin{theorem}
Let $g \in G(\cK_{>0})$. Then there exists a unique split maximal torus of $G$ containing $g$. In particular, $g$ is regular and semi-simple. 
\end{theorem}

\begin{rmk} Lusztig's theorem was originally stated over the real numbers $\R$. However, using that the field of Puiseux series over $\R$ is a real closed field, we can invoke the Artin-Schreier theorem to conclude that the eigenvalues of $g \in G(\cK_{>0})$ are Puiseux series with positive leading coefficient. Analyzing the characteristic polynomial of $g$ allows us to conclude that these eigenvalues are in fact power series.
\end{rmk}

So by the above, we can conjugate $\rho(\gamma)$ to $H(\cK_{>0})$, and then take valuations to get an element of the coweight lattice $\Lambda$. However the conjugation of $\rho(\gamma)$ into $H(\cK_{>0})$ is only well-defined up to the action of the Weyl group $S_n$. Thus we actually get a well-defined element of the dominant coweights $\Lambda^+$ which we will denote $d(\gamma)$. Let $l$ be the lamination corresponding to $x$. Then $d(\gamma)$ should be viewed as the ``length" of the lamination associated to our along the loop $\gamma$. This analogy will be explored further in the next section.

We can state a version of theorem 7.6 of \cite{FG1} for the semi-field $\cK_{>0}$. Let $S$ be a surface, with or without boundary, with $\chi(S)<0$. 
Let $\gamma$ be a non-trivial loop on $S$, not homotopic to a boundary component of $S$. 

Denote by $S'$ the surface obtained by cutting $S$ along $\gamma$. We assume that $S'$ is connected. It has two boundary components, $\gamma_+$ and 
$\gamma_-$, whose orientations are induced by the one of $S'$.  The surface $S'$  has one or two components, 
each of them of negative Euler characteristic. Denote by $$\X_{G,S'}(\gamma_+, \gamma_-)(\cK_{>0})$$
the subspace of $\X_{G,S'}(\cK_{>0})$  given by the following condition: the monodromies along $\gamma_+$ and $\gamma_-$ are inverse, and the framing at the boundary is given by the unique attracting flags for $\gamma_+$ and $\gamma_-$.

\begin{theorem}
Let $S$ be a surface with $\chi(S)<0$, and $S'$ is obtained by cutting along a loop $\gamma$, as above. 
Then the restriction from $S$ to $S'$ provides us a principal 
$H(\cK_{>0})$-bundle

$$\X_{G,S}(\cK_{>0}) \longrightarrow \X_{G,S'}(\gamma_+, \gamma_-)(\cK_{>0}).$$

\end{theorem}

Thus we have 

\begin{enumerate}
\item We may restrict $G(\cK_{>0})$-bundles on $S$ to obtain $G(\cK_{>0})$-bundles on $S'$.

\item The image of the restriction map consists of $G(\cK_{>0})$-bundles on $S'$ satisfying the 
constraint that the monodromies around the oriented loops $\gamma_+$ and $\gamma_-$ are opposite, and that the flags attached to these boundaries are the unique attracting flags. 

\item Given a $G(\cK_{>0})$-bundle on $S'$ satisfying this monodromy constraint on the boundaries, one can glue it to a $G(\cK_{>0})$-bundle on $S'$, and the group $H(\cK_{>0})$ acts simply transitively on the set of the gluings. 
\end{enumerate}

The proof of the above theorem relies on an analysis of configurations of flags paramaterized by the cyclic set at infinity. In particular, it relies on analysis of the relationship between the cyclic sets at infinity of $S$ and $S'$. The proof is no different in the $\R_{>0}$ and $\cK_{>0}$ cases.

For $S$ a surface surface with boundary, the tropical points of $\X_{G,S}(\Z^t)$ arise as equivalence classes of points of $\X_{G,S}(\cK_{>0})$. Our notion of equivalence was based on taking valuations of a preferred set of functions on $\X_{G,S}(\cK_{>0})$, namely the various cluster co-ordinate charts. For closed surfaces, there is no clear choice of a set of functions. But we have the following conjecture, which we strongly believe to be true:

\begin{conj} Let $S$ be a surface with boundary such that $\chi(S)<0$. Suppose cutting along a loop $\gamma$ gives us the surface $S'$. Then the map
$$\X_{G,S}(\cK_{>0}) \longrightarrow \X_{G,S'}(\cK_{>0}).$$
respects valuations. In other words, if we take our preferred co-ordinate charts on $\X_{G,S}(\cK_{>0})$, if two points in $\X_{G,S}(\cK_{>0})$ have co-ordinates with the same valuation (recall that having co-ordinates with the same valuation in one chart is the same as having co-ordinates of the same valuation in every chart), then their images in $\X_{G,S'}(\cK_{>0})$ will have co-ordinates with the same valuation.
\end{conj}

Then it makes sense to define higher laminations on a closed surface as follows:

\begin{definition} A higher lamination on any surface $S$ is given by equivalence classes of points in $\X_{G,S}(\cK_{>0})$. Two points are equivalent if they give equivalent laminations on any surface $S'$ with boundary obtained by cutting $S$.
\end{definition}

Note that this coincides with with the definition for surfaces with boundary assuming the above conjecture. In fact, it is probably the case that to test for equivalence between two points of $\X_{G,S}(\cK_{>0})$ it is enough to look at the induced laminations on a finite number of surfaces $S'$ obtained by cutting and not all the possible surfaces obtained by cutting.


We can now state a conjectured tropical version of theorem 7.6 of \cite{FG1}. Let $S$ be a surface, with or without boundary, with $\chi(S)<0$. 
Let $\gamma$ be a non-trivial loop on $S$, not homotopic to a boundary component of $S$. 
Denote by $S'$ the surface obtained by cutting $S$ along $\gamma$. We assume that $S'$ is connected. It has two boundary components, $\gamma_+$ and 
$\gamma_-$, whose orientations are induced by the one of $S'$.  A lamination on $S$ induces a lamination on $S'$.
The surface $S'$  has one or two components, 
each of them of negative Euler characteristic. Denote by $${\X}^+_{G,S'}(\gamma_+, \gamma_-)(\Z^t)$$
the subspace of ${\X}^+_{G,S'}(\Z^t)$  given by the following condition: The lengths of the lamination along $\gamma_+$ and $\gamma_-$ are inverse: $$d(\gamma_+) = -w_0 d(\gamma_-).$$ Because the (semi-simple part of the) monodromy around a boundary component is given by a monomial map, this subspace of laminations will be a linear subspace of the space of laminations ${\X}^+_{G,S'}(\Z^t)$ of codimension ${\rm dim} H$.

\begin{conj}
Let  $S$ be a surface with $\chi(S)<0$, and $S'$ is obtained by cutting along a loop $\gamma$, as above. 
Then the restriction from $S$ to $S'$ gives a map

$$\X_{G,S}(\Z^t) \longrightarrow \X_{G,S'}(\gamma_+, \gamma_-)(\Z^t).$$

The fibers of this map are $H(\Z^t)/W(\gamma_+)$. Here $W(\gamma_+)$ is the subgroup of the Weyl group that fixes $d(\gamma_+)$ (or equivalently $d(\gamma_-)$). 

\end{conj}

Thus we have 

\begin{enumerate}

\item We may restrict laminations on $S$ to obtain laminations on $S'$.

\item The image of the restriction map consists of laminations $S'$ satisfying the 
constraint that the lengths of the laminations around the oriented loops  
$\gamma_+$ and $\gamma_-$ are opposite. 

\item Given a lamination on $S'$ satisfying this length constraint on the boundaries, one can glue it to a lamination on $S'$, and the group $H(\Z^t)$ acts transitively on the set of the gluings, with gluings that differ by an element of $W(\gamma_+)$ giving the same gluing. 
\end{enumerate}

\section{Comparison with other works}

One application of our definition of laminations is that projectived $G$-laminations give a spherical compactification of higher Teichmuller space. This will give a Thurston-type compactification of higher Teichmuller space. We will explain this below, and compare this compactification with those found in the works of Alessandrini, Parreau, and, in the case of $G=SL_2$, the work of Morgan and Shalen. We will need to understand length functions on higher Teichmuller spaces, and study their degenerations to the boundary.

We first review how to construct this compactification. Much of this was explained in \cite{FG4}. For any positive space ${\mathcal X}$ (throughout this section, ${\mathcal X}$ will be either $\A_{G,S}$ or $\X_{G,S}$; we will assume for that $S$ is a surface with boundary), we may form its tropicalization by taking the points of ${\mathcal X}$ with values in the semifields ${\mathbb A}^t$, for ${\mathbb A} = \Z, \Q$ or $\R$. In the later two cases, we have an action of ${\mathbb A}^*_{>0}$ (the multiplicative group of positive elements) on the tropical points ${\mathcal X}({\mathbb A}^t)$. One can define this in each chart of the positive atlas and show that the action is compatible with changes of co-ordinate chart. Then we have:

\begin{definition}
Let ${\mathcal X}$ be a positive space. Let ${\mathbb A}$ be either $\Q$ or $\R$. The projectivization ${\mathbb P}{\mathcal X} ({\mathbb A}^t)$ of the tropical ${\mathbb A}$-points 
of ${\mathcal X}$ is 
$$ {\mathbb P}{\mathcal X}({\mathbb A}^t) := ( {\mathcal X}({\mathbb A}^t)- \{0\} )/{\mathbb A}^*_{>0}. $$
\end{definition}

${\mathbb P}{\mathcal X}({\R}^t)$ is a sphere. Moreover, the transition maps between co-ordinate charts on this sphere are tropical maps, and hence piecewise-linear maps, so that ${\mathbb P}{\mathcal X}({\R}^t)$ has a natural piecewise-linear structure. The set 
${\mathbb P}{\mathcal X}({\Q}^t)$ is an everywhere  dense subset of 
${\mathbb P}{\mathcal X}({\R}^t)$. 

This sphere lives at the boundary of higher Teichmuller space ${\mathcal X}(\R_{>0})$, and gives us a logarithmic compactification as in \cite{A}, \cite{P}, \cite{MS}. Let the dimension of ${\mathcal X}$ be $d$. Then in any co-ordinate chart $H_{\alpha}$, taking logarithms of the co-ordinates gives an identification of ${\mathcal X}(\R_{>0})$ with $\R^d$. Then the compactification we seek is simply the radial compactification of $\R^d$. Any point in the spherical boundary of $\R^d$ corresponds to some relative growth rates of the co-ordinates in the chart $H_{\alpha}$.

One of the main theorems of \cite{FG1} tells us that transition functions between co-ordinate charts are given by positive rational functions (in fact, they are expected to be positive Laurent polynomials). Because of this, the growth rates in any chart $H_{\alpha}$ completely determine the growth rates in any other chart $H_{\beta}$. A point $(x_1, x_2, \dots, x_d) \in H_{\alpha}(\R^t)$ of the boundary corresponds to the limit of the points $$(e^{sx_1}, e^{sx_2}, \dots, e^{sx_d})$$ as $s \rightarrow \infty$. Suppose we have another co-ordinate chart $H_{\beta}$ with $\phi_{\alpha \beta}: H_{\alpha} \rightarrow H_{\beta}$ the transition map between co-ordinate charts. Let $$(y_1, y_2, \dots, y_d)=\phi_{\alpha \beta}^t (x_1, x_2, \dots, x_d)$$ be the tropicalization of $\phi_{\alpha \beta}$ applied to $(x_1, x_2, \dots, x_d)$. Then $$\phi_{\alpha \beta} (e^{sx_1}, e^{sx_2}, \dots, e^{sx_d})$$ is asymptotic to $$(e^{sy_1}, e^{sy_2}, \dots, e^{sy_d})$$ as $s \rightarrow \infty$.

Thus the radial logarithmic compactifications in different co-ordinate charts transform tropically, and this compactification is naturally identified with ${\mathbb P}{\mathcal X}({\R}^t)$. 

Now recall that we had a map from $\X(\cK_{>0})$ to the space of laminations $\X(\Z^t)$. Now let $$x=(x_1, x_2, \dots, x_d)$$ bbe a point of $\X(\cK_{>0})$. Suppose that the $x_i$ are in fact convergent power series in $\cK_{>0}=\R((t))_{>0}$. Then for small enough $t$, the $x_i$ are positive when evaluated at $t$, and we may view $x$ as a path in $\X(\R_{>0})$. The growth rate of $x_i$ as $t$ goes to $0$ is $\val(x_i)$. Thus given any lamination $l \in \X(\R^t)$ which is non-zero, we can construct a path in higher Teichmuller space that approaches the projectization of this lamination in the boundary. Laminations which are related by the action of ${\mathbb A}^*$ approach the same point on the boundary at different speeds. Laminations measure growth rates of a path in $\X(\R_{>0})$, while projectivized laminations measure the relative growth rates of the co-ordinates. In summary, we have mapped out the relationship between valuations, growth rates and tropical points.

To summarize: projectivized $G$-laminations give a spherical boundary for higher Teichmuller space. Points in the boundary parameterize relative growth rates of paths in higher Teichmuller space that go to infinity. We note that like Thurston's compactification of Teichmuller space, our compactification has a natural action of the (higher) mapping class group (the higher mapping class group is defined as the symmetries of the cluster algebra underlying the higher Teichmuller space). This turns out to be tautological from the definition of the higher mapping class group.


We now compare this compactification with the ones given by \cite{A} and \cite{P}. The construction outlined above, due mostly to Fock and Goncharov, works in the context of positive spaces. On the other hand, the constructions of Alessandrini and Parreau work in greater generality (for example, Parreau works in the context of representation varieties of finitely generated groups, while Alessandrini works in the context of compactifications of general algebraic varieties). On the one hand, their approaches have some advantages: in addition to being quite general, Parreau's work highlights the metric convergence of symmetric spaces to buildings, while Alessandrini relates precisely the relationship between spaces of valuations and logarithmic limit sets.

On the other hand, we are able to avoid some technical arguments that they use. Moreover, to a boundary point in their compactification, they associate some $\pi_1$ action on an affine building; however, the association is fairly non-constructive, it is one-to-many (each point in the boundary may be associated to many $\pi_1$ actions on different affine buildings), and it is difficult to pin down the invariant properties of the different possible answers.

Our contribution is to identify in geometric terms what kinds of configurations in the affine building can occur and define an equivalence relation (again, in geometric terms) on configurations coming from the same boundary point. Moreover, affine buildings are large and infinite objects; we are able to give finite invariant subsets of the building that completely capture the lamination. For example, for any triangulation of the surface, we can lift the triangulation to the universal cover. Attached to each triangle is a configuration of points in the affine building. Take the convex hull of these points. The union of these convex hulls over all the triangles in our triangulation gives a subset of the affine building which (in the case of $\A$-laminations) is finite up to the action of $\pi_1$. (One can do something similar in the case of $\X$-laminations.)

Finally, we will show that our Thurston-type compactification surjects onto the compactifcation in \cite{P}, although we conjecture that the compactifications are in fact the same. The compactification found in \cite{A} and \cite{P} is very similar to ours, except that instead of radially compactifying for the cluster co-ordinate systems, they use a different set of functions. For each path $\gamma$ on $S$, they consider the different coefficients of the characteristic polynomial of the monodromy around $\gamma$. Let $f$ be any such function.

Recall that for a point in higher Teichmuller space, the monodromy of $\gamma$ lies in $G(\R_{>0})$. One can easily check that the coefficients of the characteristic polynomial of a matrix $g \in G(\R_{>0})$ are given in terms of a positive expression in the generalized minors of this matrix. Hence, the function $f$ is given by positive rational functions of the cluster co-ordinates. Because the function $f$ is positive, it can be tropicalized to give the function $f^t$. This gives the c-length functions of \cite{P}. Moreover, the expression of $f$ as a positive Laurent polynomial menas that the growth rates of cluster co-ordinates in any chart completely determine the growth rates of $f$. From this, we get first that compactifying by growth rates of cluster co-ordinates is at least as refined as compactifying by coefficients of the characteristic polynomial around all loops in $S$.

\bibliographystyle{amsalpha}

\end{document}